\documentclass[letter,11pt]{amsart}
\usepackage{amsmath,amsthm,amssymb,amsfonts,enumerate,color,hyperref,esint, graphicx, pgf,tikz}
\usepackage[latin1]{inputenc}

\newcommand{\R}{\mathbb{R}}

\newcommand{\Rn}{\mathbb{R}^n}
\newcommand{\Z}{\mathbb{Z}}
\newcommand{\N}{\mathbb{N}}

\newcommand{\M}{\mathcal{M}}
\newcommand{\A}{\mathcal{A}}

\newcommand{\h}{\mathcal{H}}
\newcommand{\sym}{\mathcal{S}^n}

\newcommand{\tr}{\operatorname{tr}}
\newcommand{\osc}{\operatorname{osc\,}}
\newcommand{\vep}{\varepsilon}

\newcommand{\dive}{\operatorname{div}}

\definecolor{ffqqqq}{rgb}{1,0,0}
\definecolor{qqqqff}{rgb}{0,0,1}

\newtheorem{thm}{Theorem}[section]
\newtheorem{prop}[thm]{Proposition}
\newtheorem{cor}[thm]{Corollary}
\newtheorem{lem}[thm]{Lemma}
\theoremstyle{definition}
\newtheorem{defn}[thm]{Definition}
\newtheorem{rem}[thm]{Remark}

\numberwithin{equation}{section}

\allowdisplaybreaks

\author[D.~Kriventsov]{Dennis Kriventsov}
    \address{Dennis Kriventsov. Department of Mathematics\\
    Rutgers University\\
110 Frelinghuysen Rd., Piscataway, NJ 08854, USA}
    \email{dennis.kriventsov@rutgers.edu}
 
\author[M.~Soria-Carro]{Mar\'ia Soria-Carro}
    \address{Mar\'ia Soria-Carro. Department of Mathematics\\
    Rutgers University\\
110 Frelinghuysen Rd., Piscataway, NJ 08854, USA}
    \email{maria.soriacarro@rutgers.edu}



\thanks{This research is supported by NSF grant DMS-2247096.}
  
\begin{document}

\title[A parabolic free transmission problem]{A parabolic free transmission problem: flat free boundaries are smooth}

\begin{abstract}
We study a two-phase parabolic free boundary problem motivated by the jump of conductivity in composite materials that undergo a phase transition. Each phase is governed by a heat equation with distinct thermal conductivity, and a transmission-type condition is imposed on the free interface. We establish strong regularity properties of the free boundary: first, we prove that flat free boundaries are $C^{1,\alpha}$ by means of a linearization technique and compactness arguments. Then we use the Hodograph transform to achieve higher regularity.
To this end, we prove a new Harnack-type inequality and develop the Schauder theory for parabolic linear transmission problems.
\end{abstract}

\maketitle

\section{Introduction}

In this paper, we consider the following two-phase parabolic free boundary problem:
\begin{equation} \label{FBP}
\begin{cases}
a_+u_t - \Delta u = 0 & \hbox{in}~\{u>0\},\\
a_-u_t - \Delta u = 0 & \hbox{in}~\{u<0\},\\
\partial_\nu u^+ =\partial_\nu u^-& \hbox{on}~ \{u=0\},\\
\end{cases}
\end{equation}
for some constants $a_+, a_->0$ with $a_+\neq a_-$, where  $u^\pm= \max(\pm u, 0)$ and
$\partial_\nu u^\pm$ denote the normal derivatives of $u^\pm$ in the inward direction to $\{\pm u>0\}$. The set $F(u):=\{u=0\}$ is known as the \textit{free boundary}, and it is our main object of study.

This model arises formally when considering solutions to the quasilinear parabolic  equation
\begin{align} \label{modeleq}
& \ \partial_t c(u) - \Delta u = 0,\\
&c(s) = a_+ s^+ - a_- s^-.\label{cu}
\end{align} 
Equations like \eqref{modeleq} model several physical phenomena involving phase transitions, such as heat conduction through composite materials and fluid flow through porous media. 
Indeed, setting $v=c(u)$, \eqref{modeleq} and \eqref{cu} transform into the divergence form equation
\begin{equation} \label{eq:divform}
    v_t - \dive\big(A(v) \nabla v \big)=0,  \quad A(v) = a_+^{-1} \chi_{\{v>0\}} + a_-^{-1} \chi_{\{v\leq 0\}}.
\end{equation}
The function $v$ represents the temperature, and the model describes the evolution of heat through materials that undergo a phase change, resulting in a jump of conductivity across a free interface (i.e., the zero-level set of $v$). This type of behavior has been observed, for instance, in single-crystalline nanowires \cite{WLSXMZ}.
The mathematical theory in the elliptic case has been developed by several authors. For varying scalar coefficients, the model was studied in \cite{KLS}, and for matrix coefficients in \cite{AM, AT, ST}; see also \cite{LN, LV} for transmission problems where the jump in conductivity occurs across a fixed surface. 

Equation \eqref{modeleq} has also been considered for $c\in C(\R)\cap C^1(0,\infty)$ satisfying
\begin{equation} \label{epc}
c(s)= 0 \quad\hbox{if}~s\leq 0 \quad \hbox{and} \quad c'(s)>0\quad\hbox{if}~s>0.
\end{equation}
This equation describes the slow flow of an incompressible fluid in a partially saturated homogeneous porous medium. In this context, $u$ represents the pressure of the fluid, and the zero-level set splits the domain into two disjoint regions: the saturated part, where the equation becomes elliptic, and the unsaturated part, where it becomes parabolic. This model falls into the scope of elliptic-parabolic equations, with significant progress in  \cite{AL,BD,BH,DBG,HW,KP}.  It is worth noting that the function  $c(s)= a_+ s^+$ satisfies \eqref{epc}, so
 taking $a_-$ arbitrarily small in \eqref{cu}, our problem can be understood as an approximating model for describing these types of processes.

There is an extensive literature on phase transition problems involving caloric functions. These models can be classified according to the type of free boundary condition. Two of the most popular ones are the so-called Stefan problems and  Bernoulli problems. 

Stefan problems describe the melting or solidification of a material with a solid-liquid interphase, such as ice and water. Typically, the evolution of this interphase is described by the equation 
$V_\nu = |\nabla u^+| - |\nabla u^-|,$
where $V_\nu$ is the speed of propagation in the normal direction $\nu$ to the free boundary. These problems, introduced by Athanasopoulos, Caffarelli, and Salsa in the seminal papers \cite{ACS,ACS1,ACS2}, have been studied by many authors. See \cite{DSFS} for recent developments.

Parabolic Bernoulli problems appear, for instance, in combustion, and were originally derived as a singular limit from a model describing the propagation of laminar flames with high activation energy \cite{CV, CLW}. In this case, the interaction at the free front is given by
$|\nabla u^+|^2 - |\nabla u^-|^2=1.$
See also \cite{AW,D,W} for related works in this direction.

We point out that the free boundary condition given in \eqref{FBP}, equivalently expressed as
$$
|\nabla u^+|= |\nabla u^-|,
$$
differs qualitatively from the two previously discussed. Here, we prescribe a transmission-type condition that characterizes the flux balance across the zero-level set of the solution. In this case, the possible singularities arise due to the jump discontinuity of the global operator across the free interface.
Such problems are commonly referred to as \textit{free} transmission problems, and have received notable attention over the past decade,
particularly within the framework of elliptic equations \cite{AM,AT,CDSS,KLS,PS,PSw,ST}.
However, the literature concerning their parabolic counterparts is relatively limited.
The authors anticipate that the methods developed in this paper could be useful in studying other parabolic problems of the same nature. 

The analysis of free boundaries has been a central topic in PDEs over the past 40 years. The pioneer works of Caffarelli \cite{C1,C2}, for the two-phase elliptic Bernoulli problem, provide a well-established program to approach the free boundary regularity, consisting of two steps: I. Lipschitz free boundaries are $C^{1,\alpha}$; and II. Flat free boundaries are Lipschitz. These results rely on monotonicity formulas and Harnack principles; see also \cite{W1,W2,CFS} for related extensions.
De Silva \cite{DS} developed a different method to establish that \textit{flat free boundaries are $C^{1,\alpha}$}. Her techniques are very flexible and can be applied to many problems, including equations with variable coefficients and distributed sources.
 
In this paper, we consider viscosity solutions of \eqref{FBP}. 
This notion has been successfully used to study nonlinear elliptic and parabolic problems with the presence of a free boundary \cite{ACS,ACS1,ACS2,CSa,DSFS}.
The distinguishable feature of our model is that the heat equations governing each phase have distinct conductivity coefficients (i.e., $a_+\neq a_-$), forcing the Laplacian of $u$ to jump across this surface. This is in contrast to the parabolic Bellman equations, such as $u_t = \max(\Delta u/a_+, \Delta u /a_-)$, for which solutions are known to be $C^{2,\alpha}$ \cite{EL}.

About the optimal regularity of solutions, it is known that $u$ might not be Lipschitz \textit{in time.} This was proved by Caffarelli and Stefanelli  \cite{CS}, who produced a counterexample satisfying  \eqref{modeleq} and \eqref{cu}, with $a_+=1$ and $a_-=2$. It is natural to ask whether there are solutions that also fail to be Lipschitz \textit{in space}. For general coefficients $a_+$ and $a_-$, this problem seems fairly challenging due to the lack of symmetry and monotonicity formulas. 
Nonetheless, we were able to show that non-Lipschitz special solutions exist for $a_+=1$ and $a_-$ sufficiently small.
We give further details in Section~\ref{sec:counterexample}.

For divergence-form elliptic equations with different \textit{scalar} coefficients on each phase, Kim, Lee, and Shahgholian \cite{KLS} established the Lipschitz continuity of weak solutions by means of a new ACF-type monotonicity formula. However, their techniques do not seem to work for general matrix coefficients or even time-dependent equations like ours. In fact, it is not true in general that solutions must be Lipschitz in the case of matrix coefficients, and it is an open problem to characterize the set of matrices for which the result holds. For more details about this problem, see the papers by Andersson and Mikayelyan \cite{AM}, and Soave and Terracini \cite{ST}; see also a related work by Caffarelli, De Silva, and Savin \cite{CDSS}, who establish Lipschitz continuity for an anisotropic model in dimension 2.

It is worth noting here that for the elliptic free transmission context, the case of constant scalar coefficients (i.e., looking at steady-state solutions of \eqref{eq:divform}) is essentially trivial: performing the inverse change of variables $u = c^{-1}(v)$ results in a harmonic function $u$ with the same level set as $v$, so $F(v)$ has precisely the structure of a level set of a harmonic function. In the parabolic setting, this change of variables only passes between the two formulations \eqref{eq:divform} and \eqref{FBP}. Our analysis here suggests that these models are closer in difficulty to the matrix-coefficient elliptic problem, albeit with different structure.

Our main goal is to investigate the regularity of the free boundary $F(u)$ under the assumption that the graph of $u$ is \textit{flat} in some direction. Note that this is the expected behavior of $u$ at regular free boundary points, i.e., points where the transmission condition is satisfied in the classical sense. 
More precisely, we show that if the graph of $u$ is trapped between two close parallel planes, then the free boundary must be locally smooth. The proof is divided into two main steps. First, we prove that the free boundary is locally the graph of a $C^{1,\alpha}$ function (using the parabolic topology). This will be achieved via an improvement of flatness iteration procedure based on a Harnack-type inequality. Our methods are inspired by the work of De Silva \cite{DS} for the one-phase elliptic Bernoulli problem.

 Second, we show that $C^{1,\alpha}$ free boundaries are smooth in space and time. Thanks to the first step, we can apply the Hodograph transform to the function $u$ in a neighborhood of a free boundary point. This powerful tool, introduced by Kinderlehrer, Nirenberg, and Spruck \cite{KN, KNS} in the context of free boundary problems, allows us to write \eqref{FBP} as a nonlinear transmission problem in the hodograph variables. In particular, this transformation maps $F(u)$ to a portion of a hyperplane, and its regularity follows from the boundary Schauder estimates for solutions to these types of problems that we derive in Section~\ref{sec:transmission}.

The study of parabolic problems introduces new difficulties compared to their elliptic counterparts. On the one hand, constructing barriers to establish Harnack principles becomes notably challenging. For elliptic Bernoulli or free transmission problems, a typical approach involves applying the interior Harnack inequality in a small ball contained on one side and employing a suitable barrier to propagate this estimate across the free boundary. One could try a similar strategy for parabolic problems, but has to deal with the additional difficulty of propagating the barrier for future times.
Adapting the elliptic methods to this \textit{anisotropic} parabolic problem is one of the main novelties of this work.  

On the other hand, the parabolic topology presents some subtleties, especially to gain regularity of the free boundary in time. A general strategy to prove smoothness is to carry out a bootstrap argument by differentiating the equations and gaining one derivative at each step. This argument works if one has sufficient initial regularity.   
Starting from $C^{1,\alpha}$ is enough to gain regularity in space, but not in time, since we only know that the free boundary is  $C_t^{0,(1+\alpha)/2}$. We provide more details in Section~\ref{sec:hodograph}.

\medskip

Our main results are the following. 

\begin{thm}[Smoothness]\label{thm1}
Let  $u \in C(Q_1)$  be a viscosity solution of \eqref{FBP}.
There exists $\bar \delta >0$, depending on $n$, $a_+$, and $a_-$, such that if the graph of $u$ is $ \delta$-flat in the $e_n$-direction in $Q_1=B_1\times (-1,0]$, i.e.,
\begin{equation} \label{flatness}
\sup_{Q_1} |u(x,t) - x_n | \leq  \delta \quad \hbox{for some}~\delta \leq \bar \delta,
\end{equation}
 then $F(u)$ is locally smooth in space and time.
\end{thm}

Theorem~\ref{thm1} will follow from the next result. 

\begin{thm}[$C^{1,\alpha}$ regularity] \label{thm2}
Suppose we are under the assumptions of Theorem~\ref{thm1}.
There exist $\alpha \in (0,1)$ and $C_0>0$, depending on $n$, $a_+$, and $a_-$, such that $F(u) \cap Q_{1/2}$ is a $C^{1,\alpha}$ graph in the $e_n$-direction, i.e., there is some $C^{1,\alpha}$ function $g: B_{1/2}' \times (-1/4,0] \to (-\delta, \delta)$ such that 
$$
F(u) \cap Q_{1/2}= \big\{ (x, t) \in Q_{1/2} : u(x,t)=0 \ \hbox{and} \  x_n = g(x',t) \big \}
$$
and 
$
\| g \|_{C^{1,\alpha}({B_{1/2}' \times (-1/4,0]})} \leq C_0.
$
\end{thm}

This paper is organized as follows. In Section~\ref{sec:prelim}, we introduce the necessary notation and definitions that we will use. In Section~\ref{sec:harnack}, we show a Harnack-type inequality and derive useful H\"{o}lder-type estimates. 
In Section~\ref{sec:transmission}, we consider linear flat transmission problems and establish the regularity of viscosity solutions at the interface. In Section~\ref{sec:improvement}, we show the improvement of flatness result via a compactness argument, which relies on the Harnack inequality and the regularity of solutions to the limiting transmission problem. We prove Theorem~\ref{thm2} in Section~\ref{sec:fbregularity} and
 Theorem~\ref{thm1} in Section~\ref{sec:hodograph}. We give a counterexample to Lipschitz regularity of solutions in Section~\ref{sec:counterexample}. In the Appendix, we prove a stability result for compact graphs in metric spaces with the Hausdorff distance.  
 
 \section{Preliminaries} \label{sec:prelim}

 We introduce some notation and definitions that we will use throughout this paper.

 \subsection{Notation} 
Given $(x,t)\in \R^{n+1}$, we call $x\in \Rn$ the space variable and $t \in \R$  the time variable. A point $x\in \Rn$ will sometimes be written as $x=(x',x_n)$, where $x'\in \R^{n-1}$ and $x_n \in \R$.
Given a function $u(x,t)$, we denote by $\nabla u$ the gradient of $u$ with respect to $x$, $D^2 u$ the Hessian with respect to $x$, $\nabla ' u$ the gradient with respect to $x'$, $u_n$ the partial derivative with respect to $x_n$, and $u_t$ the partial derivative with respect to $t$.

 We call a \textit{parabolic cylinder} any set of the form 
 $$Q=\Omega\times (t_1, t_2],$$ 
 where $\Omega$ is a smooth bounded domain in $\Rn$ and $t_1<t_2$. We denote by $\partial_p Q$ the \textit{parabolic boundary} of $Q$, i.e., $\partial_p Q= \partial_b Q \cup \partial_l Q$, where
 $ \partial_b Q= \Omega \times \{t_1\}$ is the \textit{bottom}, and
 $\partial_l Q= \partial \Omega \times [t_1,t_2]$ is the \textit{lateral} part of the boundary.
  
Given $(x,t), (y,s)\in \R^{n+1}$, we define the \textit{parabolic distance} as
$$
d_p((x,t), (y,s)) :=\big(|x-y|^2 + |t-s|\big)^{1/2}.
$$
 We denote Euclidean balls and cylinders in the usual way. Let $B_r = \{x \in \Rn : |x|<r\}$ be the ball of radius $r$ centered at $0$ in $\R^n$. For $x_0\in \Rn$ and $t_0 \in \R$, we write: 
 \begin{align*}
   B_r(x_0) &= B_r + x_0,\\
  Q_r(x_0,t_0) &= B_r(x_0)\times(t_0-r^2,t_0], \\
   Q_r^\pm(x_0,t_0) &=   Q_r(x_0,t_0) \cap \{\pm x_n>0\}.
 \end{align*}
For $(x_0,t_0)=(0,0)$, we  write $Q_r$ and $Q_r^\pm$.
  Let $u$ be a continuous function in $Q$. We denote the positive and negative phases of $u$ in $Q$ by  $Q^\pm(u) = \{\pm u>0\} \cap Q$, and the free boundary by $F(u) = \{u=0\} \cap Q$.  If $Q=Q_r$, with $r>0$ and $r\neq 1$, we write $Q_r^\pm(u)$ and $F_r(u)$. 
 
 \subsection{Function spaces}
Let $C(Q)$ be the set of continuous functions in $Q$. 
Given $k\in \N$, $C^{k}(Q)$ is the set of all functions $u$ such that $D_x^m D_t^l u \in C(Q)$ for all $|m|+2 l\leq k$,  where $m$ is a multi-index and $l\geq 0$ is an integer.
 
 Given $\alpha\in(0,1]$, we define $C^{0,\alpha}(\overline{Q})$ as the set of all functions $u\in C(\overline{Q})$ such that
 $$
\|u\|_{C^{0,\alpha}(\overline{Q})} := \|u\|_{L^\infty(Q)} + [u]_{C^{0,\alpha}(\overline{Q})} < \infty,
 $$
 where the $\alpha$-H\"{o}lder seminorm (or Lipschitz seminorm if $\alpha=1$) of $u$ is given by
 \begin{align*}
 [u]_{C^{0,\alpha}(\overline{Q})} := \sup_{ (x,t)\neq (y,s)} \frac{|u(x,t)-u(y,s)|}{d_p((x,t),(y,s))^{\alpha}}.
 \end{align*}
 We also define the $\alpha$-H\"{o}lder seminorm in $t$ as
 $$
[u]_{C_t^{0,\alpha}(\overline{Q})} := \sup_{ (x,t)\neq (x,s)} \frac{|u(x,t)-u(x,s)|}{|t-s|^\alpha}.
 $$
Note that if $u\in C^{0,\alpha}(\overline Q)$, then $u$ will be $\alpha$-H\"{o}lder continuous in $x$ and $\alpha/2$-H\"{o}lder continuous in $t$.  More generally, we define $C^{k,\alpha}(\overline{Q})$ as the set of all functions $u\in C(\overline{Q})$ such~that
 \begin{align*}
      \|u\|_{C^{k,\alpha}(\overline{Q})} &:= \sum_{|m| + 2 l \leq k} \|D_x^m D_t^l u\|_{L^\infty(Q)} + \sum_{|m| + 2 l = k} [D_x^m D_t^l u]_{C^{0,\alpha}(\overline{Q})} \\
      &\qquad\qquad +  \sum_{|m| + 2 l = k-1} [D_x^m D_t^l u]_{C_t^{0,(1+\alpha)/2}(\overline{Q})} <\infty. 
 \end{align*}
The spaces $C^{k,\alpha}(\overline{Q})$ can be characterized in terms of pointwise properties using Campanato's definition. This characterization will be useful to show boundary regularity estimates in Section~\ref{sec:transmission}. 
Given an integer $k \geq 0$, a \textit{$k$-th order polynomial} is a function  of the form
 $$
 P_k(x,t) = \sum_{|m| + 2 l \leq k} c_{m l} x^m t^l,
 $$
 for some coefficients $c_{ml}\in \R$, where $m$ and $l$ are as before.
 Suppose that $(0,0)\in Q$.
We say that a function $u\in C(Q)$ is \textit{pointwise} $C^{k,\alpha}$ at $(0,0)$, and write $u\in C^{k,\alpha}(0,0)$, if 
$$
[u]_{C^{k, \alpha}(0,0)} := \inf_{P_k} \sup_{Q_r\subset Q} {|u-P_k|}{r^{-(k+\alpha)}}<\infty.
$$
 For any point $(x,t)\in Q$, we define $[u]_{C^{k,\alpha}(x,t)} := [u(\cdot+x,\cdot+t)]_{C^{k,\alpha}(0,0)}$. 
 Then $u\in C^{k,\alpha}(\overline{Q})$ if and only if $D_x^m D_t^l u\in C(\overline Q)$ for all $|m| + 2 l \leq k$, and
  $$
 \sup_{(x,t)\in \overline{Q}} [u]_{C^{k,\alpha}(x,t)} < \infty. 
 $$
For instance, see \cite[Lemma 2.3]{LW}.

\subsection{Viscosity solutions}

In the following definitions, $Q$ will be a parabolic cylinder.   
 
 \begin{defn}
Given $u, v \in C(Q)$, we say that $v$ touches $u$ from above (resp. below) at $(x_0,t_0)\in Q$ if $u(x_0,t_0)=v(x_0,t_0)$, and there is some $r>0$ such that
 $$u(x,t) \leq v(x,t) \quad (\hbox{resp.}~\geq) \quad \hbox{for all}~ (x,t)\in Q_r(x_0,t_0) \subset Q.$$
 \end{defn}

 \begin{defn} \label{viscsol}
 We say that $u\in C(Q)$ is a \textit{viscosity subsolution} (resp.~supersolution) to the free transmission problem \eqref{FBP} if the following conditions are satisfied:
\begin{enumerate}[$(i)$]
\item If $v\in C^{2}(Q^+(u))$ touches $u$ from above (resp.~below) at $(x_0,t_0)\in Q^+(u)$, then
 $$a_+v_t(x_0,t_0) - \Delta v (x_0,t_0) \leq 0 \quad (\hbox{resp.}~\geq).$$
 
 \item If $v\in C^{2}(Q^-(u))$ touches $u$ from above (resp.~below) at $(x_0,t_0) \in Q^-(u)$, then
 $$a_-v_t(x_0,t_0) - \Delta v (x_0,t_0) \leq 0 \quad (\hbox{resp.}~\geq).$$
 
\item  If $v\in C^{2}(\overline{Q^+(v)})\cap C^{2}(\overline{Q^-(v)})$ touches $u$ from above (resp.~below) at $(x_0,t_0) \in F(u)\cap F(v)$ and $|\partial_\nu v^\pm (x_0,t_0)|> 0$, then 
$$
\partial_\nu v^- (x_0, t_0) - \partial_\nu v^+ (x_0, t_0) \leq 0 \quad (\hbox{resp.}~\geq),
$$
where $\partial_\nu v^\pm (x_0,t_0)=- |\nabla v^\pm (x_0,t_0)|$ and
$$
\nabla v^\pm (x_0,t_0) = \lim_{\substack{(y,s)\to (x_0,t_0)\\ (y,s)\in \{\pm v>0\}}} \nabla v^\pm (y,s).
$$
\end{enumerate}
We call $u$ a \textit{viscosity solution} if it is both a viscosity subsolution and a viscosity supersolution.
 \end{defn}

\section{Harnack Inequality} \label{sec:harnack}

We prove a Harnack-type inequality (Theorem~\ref{Harnack}) for viscosity solutions of \eqref{FBP} in $Q_1$. As a consequence, we derive a key regularity estimate (Corollary~\ref{holder}) for our compactness argument in Section~\ref{sec:improvement}.
Theorem~\ref{Harnack} will follow from iterating the next main lemma.

\begin{lem} \label{lem:harnack}
There exists a universal constant $\delta_0>0$ such that if $u$ is viscosity solution to \eqref{FBP} in $Q_1$ satisfying
$$
x_n-\delta \leq u(x,t) \leq x_n + \delta \quad \hbox{for some}~\delta \leq \delta_0,
$$
and $u(e_n/2, -9/10)\geq x_n$, then there exists $c\in (0,1)$, depending on $n$, $a_+$, and $a_-$, such that
$$
u(x,t)\geq x_n -(1-c)\delta \quad  \hbox{in}~{Q_{1/3}}.
$$
Analogously, if $u(e_n/2, -9/10)\leq x_n$, then
$$
u(x,t) \leq x_n + c \delta \quad \hbox{in}~{Q_{1/3}}.
$$
\end{lem}

To prove this result, we need the classical Hopf Lemma for uniformly parabolic equations; see \cite[Lemma 2.8]{L}. For clarity, we state it here in the form we will use it.

\begin{lem}[Hopf Lemma] \label{Hopf}
For $\eta, \rho>0$ and $(y,s)\in \R^{n+1}$, let 
$$E\equiv E_{\eta,\rho}(y,s) = \big\{(x,t) \in \R^{n+1} : |x-y|^2+\eta^2 (s-t) < \rho^2 , \ t< s \big \}.$$
Suppose that $u\in C^2(E)$ satisfies $a u_t - \Delta u \geq 0$ in $E$ for some $a>0$, and that there is $(x_0, s)\in \R^{n+1}$ with $|x_0-y|=\rho$ such that 
\begin{enumerate}[$(i)$]
    \item $u(x_0,s) \leq u(x, t)$ for all $(x,t)\in E$;
    \item $u(x_0, s) < u(x,t)$  for all $(x,t) \in E$  such that $|x-y| \leq \rho/2.$
\end{enumerate}
If $\nu = \frac{y-x_0}{|y-x_0|}$, then $\partial_\nu u (x_0, s) >0$. 
\end{lem}

We give the proof of the main lemma.

\begin{proof}[Proof of Lemma~\ref{lem:harnack}]
We may assume that $a_+, a_- \in (0,1]$. Otherwise, we take $a u(x/a, t/a)$, where $a=\max(a_+,a_-)$.
Fix $r=5/12$ and $\delta_0\in(0, 1/20]$. Let $u$ satisfy
\begin{equation}\label{assumption}
x_n -\delta\leq u \leq x_n + \delta \quad \hbox{in}~Q_1
\end{equation}
for some $\delta \in (0, \delta_0]$. Then the free boundary of $u$ is trapped between the planes $x_n=-\delta$ and $x_n=\delta$, and thus,  $B_{r}(e_n/2) \subset \{u(\cdot, t)>0\}$ for any $t\in (-1,0]$. Set $u_\delta = \frac{u-x_n}{\delta}$. By \eqref{assumption}, we have $u_\delta +1\geq 0$ in $Q_1$ and $(u_\delta+1)(e_n/2,-9/10)\geq 1$, by assumption.

Applying the interior Harnack inequality for the heat equation in $B_{r}(e_n/2)$, we see that there exists some universal constant $c_0\in (0,1)$ such that
\begin{equation*} 
\inf_{x\in B_{r}(e_n/2)} u_\delta(x,-4/5) \geq -1+ c_0,
\end{equation*}
or equivalently, 
\begin{equation}\label{interiorH1}
    u(x,t) \geq x_n -(1-c_0)\delta \quad \hbox{for all}~(x,t)\in B_r(e_n/2) \times \{-4/5\}.
\end{equation}

For $t\in[-4/5,0]$, consider $s(t) = e^{-K\big(t+\tfrac{4}{5}\big)}$, where $K>0$ is some large constant to be determined. Let $B_{r}'=\{x' \in \R^{n-1} : |x'|<r\}$. Define the function $\phi: \overline{B_{r}'} \times [-r, r] \to [0, 1]$ as  $$\phi(x',x_n)=\phi_1(x')\phi_2(x_n),$$ 
where $\phi_1 : \overline{B_r'}\to[0,1]$ is the first  eigenfunction of the Laplacian in $B_{r}'$ corresponding to the first eigenvalue $\lambda_1>0$, with zero boundary data, i.e., 
$$
\begin{cases}
- \Delta_{x'} \phi_1 = \lambda_1 \phi_1 & \hbox{in}~B_{r}',\\
\phi_1 = 0 & \hbox{on}~\partial B_{r}',
\end{cases}
$$
such that $\sup_{B_{r'}}\phi_1=1$, and $\phi_2 : [-r,r]\to[0,1]$ is the bell-shaped even function
$$
\phi_2(x_n) = \frac{2}{r^3} |x_n|^3 - \frac{3}{r^2} x_n^2+1.
$$
Consider the oblique cylinder (see Figure~\ref{fig:barrier}): 
 $$D = \displaystyle\bigcup_{t\in (-4/5, 0]} \Big(B_{r}'\times\big[-\tfrac{5}{8}t-r, -\tfrac{5}{8} t+ r\big]\Big)\times \{t\}.$$
The parabolic boundary of $D$, $\partial_pD$, is the union of the lateral boundary 
$$\partial_l D= \displaystyle\bigcup_{t\in [-4/5, 0]} \big\{(x',x_n)\in B_1 : |x'|=r, \  |x_n + \tfrac{5}{8}t|=r\big\} \times \{t\}$$
and the bottom 
$
\partial_b D = \Big(B_{r}'\times\big(\tfrac{1}{2}-r, \tfrac{1}{2}+r\big)\Big) \times \{-4/5\}.
$
For $(x,t)\in \overline{D}$, define the function
$$
w(x,t) = x_n-\delta + c_0 \delta s(t) \phi\big(x',x_n+\tfrac{5}{8}t\big) - c_0\delta.
$$ 

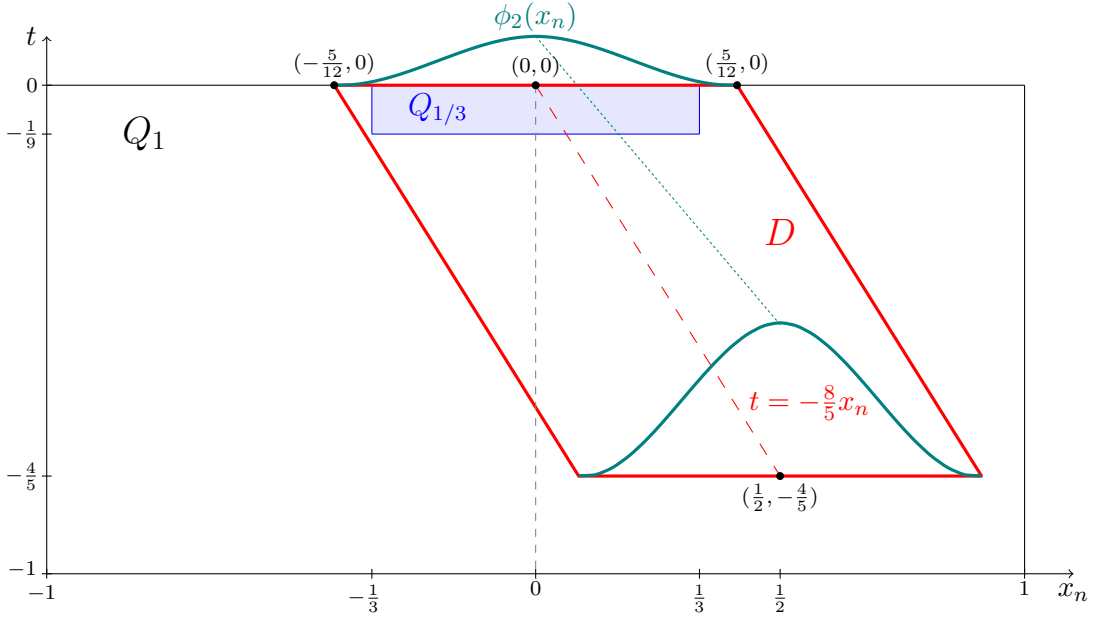
\begin{figure}[h] \label{fig:barrier}
\begin{tikzpicture}[domain = -12:12,scale = 0.65, x=1.0cm,y=1.0cm]
\fill[color=qqqqff,fill=qqqqff,fill opacity=0.1] (-3.351,0.) -- (-3.35,-1.) -- (3.35,-1.) -- (3.35,0.) -- cycle; 
\draw (-10.,0.)-- (-10.,-10.);
\draw (-10.,-10.)-- (10.,-10.);
\draw (10.,-10.)-- (10.,0.);
\draw (10.,0.)-- (-10.,0.);
\draw[color=black] (-8,-1) node {\Large$Q_1$};
\draw [color=qqqqff] (-3.35,0.)-- (-3.35,-1.);
\draw [color=qqqqff] (-3.35,-1.)-- (3.35,-1.);
\draw [color=qqqqff] (3.35,-1.)-- (3.35,0.);
\draw [color=qqqqff] (3.35,0.)-- (-3.35,0.);
\draw[color=qqqqff] (-2,-0.5) node {\large$Q_{1/3}$};
\draw [line width=1.2pt,color=ffqqqq] (-4.12,0)-- (0.88,-8.);
\draw [line width=1.2pt,color=ffqqqq] (0.86,-8.)-- (9.14,-8.);
\draw [line width=1.2pt,color=ffqqqq] (9.12,-8.)-- (4.12,0);
\draw [line width=1.2pt,color=ffqqqq] (4.12,0)-- (-4.12,0);
\draw[color=ffqqqq] (5,-3) node {\Large$D$};
\draw [dash pattern=on 5pt off 5pt,color=ffqqqq] (0.,0.)-- (5,-8);
\draw [dash pattern=on 3pt off 3pt, color=gray] (0,0)-- (0,-10); 
\draw[color=ffqqqq] (5.6,-6.5) node {$t=-\frac{8}{5} x_n$};
\draw [dash pattern=on 1pt off 1pt, color=teal] (0,1)-- (5,-4.9);
\draw[->] (-10,0)--(-10,1) node[left] {$t$};
\draw[->] (10,-10)--(11,-10) node[below] {$x_n$};
\draw[scale=1,domain=-4.12:4.12,smooth,color=teal, very thick]
  plot (\x, {1/32* abs(\x)^3 - 3/16* (\x)^2+1}); 
 \draw[color=teal] (0, 1.4) node {$\phi_2(x_n)$};
  \draw[scale=1,domain=0.88:9.12,smooth,color=teal,very thick]
  plot (\x, {(1/32* abs(\x-5)^3 - 3/16* (\x-5)^2+1)*exp(8/7)-8}); 
\begin{scriptsize}
\draw [fill=black] (-10.1,0)--(-9.9,0);
\draw[color=black] (-10.3,0) node {$0$};
\draw [fill=black] (-10.1,-10)--(-9.9,-10);
\draw[color=black] (-10.5,-9.9) node {$-1$};
\draw [fill=black] (-10.1,-8)--(-9.9,-8); 
\draw[color=black] (-10.5,-8) node {$-\frac{4}{5}$};
\draw [fill=black] (-10.1,-1)--(-9.9,-1); 
\draw[color=black] (-10.5,-1) node {$-\frac{1}{9}$};
\draw [fill=black] (5,-10.1)--(5,-9.9);
\draw[color=black] (5,-10.5) node {$\frac{1}{2}$};
\draw [fill=black] (-3.35,-10.1)--(-3.35,-9.9);
\draw[color=black] (-3.5,-10.5) node {$-\frac{1}{3}$};
\draw [fill=black] (3.35,-10.1)--(3.35,-9.9);
\draw[color=black] (3.35,-10.5) node {$\frac{1}{3}$};
\draw [fill=black] (0.,0.) circle (2pt) node[above] {$(0,0)$};
\draw [fill=black] (-4.12,0.) circle (2pt) node[above] {$(-\frac{5}{12},0)$};
\draw [fill=black] (4.12,0.) circle (2pt) node[above] {$(\frac{5}{12},0)$};
\draw [fill=black] (5,-8) circle (2pt)  node[below] {$(\frac{1}{2},-\frac{4}{5})$};
\draw [color=black] (0,-10.1)--(0,-9.9);
\draw[color=black] (0,-10.3) node {$0$};
\draw [color=black] (10,-10.1)--(10,-9.9);
\draw[color=black] (10,-10.3) node {$1$};
\draw [color=black] (-10,-10.1)--(-10,-9.9) ;
\draw[color=black] (-10.1,-10.4) node {$-1$};
\end{scriptsize}
\end{tikzpicture}
\caption{Illustration of the domain $D$ and the function $s(t)\phi_2\big(x_n + \tfrac{5}{8}t\big)$.}
\end{figure}

Note that $w\in C^2(D)$. Furthermore, $0< s(t) \leq 1$ for all $t\in[-4/5, 0]$, $0 \leq \phi\leq 1$ on $\overline{D}$, and $\phi\big(x',x_n+\tfrac{5}{8}t\big)=0$ for all $(x',x_n)\in \partial_l D$. Hence,
$w\leq x_n - \delta$ on $\overline{D}$, so
$$
w \leq u \quad \hbox{on}~\overline D.
$$

Let $h\geq 0$ be the largest number such that $w_h=w+h$ lies below $u$ on $\overline D$, and let $(x_0, t_0)\in \overline{D}$ be a first touching point, i.e.,
$$
w_h(x_0, t_0)= u(x_0, t_0) \quad \hbox{and} \quad w_h \leq u \ \hbox{on}~\overline D.
$$
We need to show that $h \geq c_0 \delta.$ 
Assume that $h < c_0\delta$. 
We will prove that, under this assumption, such a touching point does not exist, yielding a contradiction.
We have that
\begin{align*}
w_h(x,t) < x_n - \delta \leq u(x,t) & \quad \hbox{on}~\partial_l D,\\
w_h(x,t) < x_n -(1-c_0) \delta \leq u(x,t) & \quad \hbox{on}~\partial_b D,
\end{align*}
where the last inequality follows from \eqref{interiorH1}.
Hence, $(x_0, t_0)\notin \partial_p D$.

For $(x,t)\in D$, we compute
\begin{align*}
\frac{1}{c_0\delta} \big(a_\pm \partial_t w_h - \Delta w_h \big) &= a_\pm s' \phi + \frac{5}{8} a_\pm s \phi_1 \phi_2' - s (\Delta_{x'} \phi_1) \phi_2 - s \phi_1 \phi_2'' \\
&= s\phi\left( - a_\pm K  +\lambda_1 + \frac{5}{8} a_\pm  \frac{ \phi_2' }{\phi_2} - \frac{ \phi_2'' }{\phi_2}\right),
\end{align*}
where $s=s(t)$, $\phi_1=\phi_1(x')$, and $\phi_2=\phi_2(x_n+\frac{5}{8} t)$.
If $x_n+\frac{5}{8} t\in (-r,-\tfrac{3r}{4}) \cup (\tfrac{r}{2}, r)$, then $$\frac{5}{8} a_\pm  \frac{ \phi_2' }{\phi_2} - \frac{ \phi_2'' }{\phi_2} <0,$$
since $a_+,a_-\in (0,1]$. Hence, choosing $K> \lambda_1/\min(a_+,a_-)$, get that
\begin{equation}\label{strictsub}
a_\pm \partial_t w_h - \Delta w_h < 0.
\end{equation}
If $x_n+\frac{5}{8} t \in [-\tfrac{3r}{4}, \tfrac{r}{2} ]$, then $\phi_2(x_n+\frac{5}{8} t) \geq \phi_2(-\tfrac{3r}{4} )>0$, and thus,  
$$
\left| \frac{5}{8} a_\pm  \frac{ \phi_2' }{\phi_2} - \frac{ \phi_2'' }{\phi_2}\right|\leq \frac{\|\phi_2\|_{C^2(-r,r)}}{\phi_2(-\tfrac{3r}{4})}.
$$ 
Hence, choosing $K$ large enough, we conclude \eqref{strictsub} for all points in $D$. Therefore, using that $u$ is a viscosity solution, by $(i)$ and $(ii)$ in Definition~\ref{viscsol}, we must have that $(x_0, t_0) \in F(u)\cap D$.

Since $w_h \in C^2(D)$, $w_h < u$ in $D\setminus F(u)$ and $w_h(x_0,t_0)=u(x_0,t_0)$, there exists a parabolic ball $B_\rho(x_1, t_0) \subset \{w_h >0\} \subseteq D\cap \{u>0\}$ such that
\begin{equation}\label{touch}
\partial B_\rho (x_1, t_0) \cap (F(u)\cap D) = \{(x_0, t_0)\}
\end{equation}
for some $(x_1, t_0) \in \{w_h >0\}$ and $\rho>0$ sufficiently small, i.e., $B_\rho(x_1, t_0)$ is a tangent ball to the free boundary of $u$ at the point $(x_0,t_0)$.
Let $\Omega_{\rho} = Q_{2\rho}(x_0, t_0) \cap \{ w_h >0 \}$. Consider the solution to the Dirichlet problem:
$$
\begin{cases}
a_+\varphi_t - \Delta \varphi = 0 & \hbox{in}~\Omega_{\rho},\\
\varphi = 0 & \hbox{on}~\partial_l \Omega_{\rho},\\
\varphi= g(x) & \hbox{on}~\partial_b \Omega_{\rho},
\end{cases}
$$
where $g$ is a smooth function with compact support such that $0\leq g \leq 1$. By the strong maximum principle, we have $0< \varphi< 1$ in $\Omega_\rho$. Taking $\eta>0$ sufficiently large, we see that $E_{\eta,\rho}(x_1,t_0) \subseteq \Omega_{\rho}$ and $\varphi$ satisfies the assumptions of Lemma~\ref{Hopf}. Hence, it follows that
$$
\partial_\nu \varphi (x_0, t_0) >0,
$$
where $\nu=\tfrac{x_1-x_0}{|x_1-x_0|}$ is the interior normal vector at $(x_0,t_0)$.
For $\gamma >0$, define the function 
$$\tilde w = w_h + \gamma \varphi \quad \hbox{in}~Q_{\rho}(x_0, t_0)\equiv Q,$$ where we have extended $\varphi$ by 0 in $Q \cap \{w_h \leq 0\}$. 
In particular, $\tilde w\in C^2(\overline{Q^+(\tilde w)})\cap C^2(\overline{Q^-(\tilde w)}).$
By \eqref{touch} and the fact that $w_h$ cannot touch $u$ in $D \setminus F(u)$, we get that
$$
\tilde w(x_0,t_0) = u(x_0,t_0) \quad \hbox{and}\quad
\tilde w \leq u \ \hbox{in}~Q,
$$
taking $\gamma$ sufficiently small. 
Moreover, $(x_0,t_0)\in F(u)\cap F(\tilde w)$, $|\partial_\nu \tilde w^\pm(x_0,t_0)|>0$, and
$$
\partial_\nu \tilde w^-(x_0,t_0) - \partial_\nu \tilde w^+(x_0,t_0) = - \gamma \partial_\nu \varphi^+(x_0,t_0) <0.
$$
Since $u$ is a viscosity solution, we obtain a contradiction with $(iii)$ in Definition~\ref{viscsol}.
Therefore, $h \geq c_0\delta$, and the desired result follows from the fact that $$w \geq x_n -(1-c) \delta>0 \quad \hbox{on}~\overline{Q_{1/3}} \subseteq D$$ 
with $c= c_0 e^{-\frac{4}{5}K} \inf_{B_{1/3}} \phi>0.$

\end{proof}

\begin{thm}[Harnack inequality] \label{Harnack}
There is a universal constant $\delta_0>0$ such that if $u$ is a viscosity solution to \eqref{FBP} in $Q_1$, and for some  $(x_0, t_0)\in Q_{1/2}$ we have
\begin{equation*} 
x_n + \alpha_0 \leq  u(x,t)\leq x_n + \beta_0 \quad \hbox{in}~ Q_r(x_0,t_0)\subset Q_1
\end{equation*}
with $ 0< \beta_0-\alpha_0 \leq \delta r$ and $\delta \in(0, \delta_0]$, 
then there exist $\alpha_1, \beta_1 \in \R$ such that 
$$
x_n + \alpha_1 \leq u(x,t) \leq x_n + \beta_1 \quad \hbox{in}~ Q_{r/3}(x_0, t_0)
$$
with $\alpha_0 \leq \alpha_1 \leq \beta_1 \leq \beta_0$ and $\beta_1- \alpha_1 \leq \bar c \delta r$, for some $\bar c \in (0,1)$ depending on $n$, $a_+$, and $a_-$.
\end{thm}

\begin{proof}
Taking $r^{-1} u(r x+ x_0, r^2 t + t_0)$ in place of $u(x,t)$, we may assume that $r=1$ and $(x_0,t_0)=(0,0)$. 
Let $\tilde \delta= \tfrac{\beta_0-\alpha_0}{2}\leq \delta$. Then
$$
x_n - \tilde \delta \leq u(x,t) - \tfrac{\beta_0+\alpha_0}{2} \leq x_n + \tilde \delta \quad \hbox{in}~Q_1.
$$
By Lemma~\ref{lem:harnack}, applied to the function $\tilde u= u - \tfrac{\beta_0+\alpha_0}{2}$, it follows that 
$$
\tilde u(x,t)\geq x_n -(1-c) \tilde \delta \quad \hbox{or}  \quad \tilde u(x,t)\leq x_n + c\tilde \delta \quad  \hbox{in}~Q_{1/3}
$$ 
for some $c\in (0,1)$ depending on $n$, $a_+$, and $a_-$. Assume the first case holds and set
$$
\alpha_1 =\tfrac{\beta_0+\alpha_0}{2} - (1-c) \tilde \delta 
\quad \hbox{and}\quad  \beta_1=\beta_0.$$ 
Then $\alpha_0 \leq \alpha_1 \leq \beta_1=\beta_0$ and $\beta_1-\alpha_1 =\big(1-\tfrac{c}{2}\big)(\beta_0-\alpha_0) \leq \bar c \delta$, where $\bar c = 1-\tfrac{c}{2}$. Moreover,
$$
x_n + \alpha_1 \leq u(x,t) \leq x_n + \beta_1 \quad \hbox{in}~Q_{1/3}.
$$

The result follows similarly from the second case.

\end{proof}

\begin{cor} \label{holder}
Let $u$ be as in Lemma~\ref{lem:harnack}. 
Set 
$u_\delta = \frac{u-x_n}{\delta}$ with $\delta\in (0, \delta_0/4]$.
There exist constants $C>0$ and $\gamma\in (0,1)$, depending on $n$, $a_+$, and $a_-$, such that if $(x_0, t_0) \in Q_{1/2}$, then
$$
|u_\delta (x,t) - u_\delta (x_0, t_0)| \leq C d_p((x,t), (x_0,t_0))^{\gamma}
$$
for all $(x,t)\in Q_{1/2} (x_0,t_0)$ such that $d_p((x,t), (x_0,t_0))\geq 2 {\delta}/{\delta_0}$.
\end{cor}

\begin{proof}
Let $(x_0, t_0)\in Q_{1/2}$, $r=1/2$, $\alpha_0=- \delta$, and $\beta_0 = \delta$. 
For $k\geq 1$, set $r_k= 3^{-k} r$. Applying the Harnack inequality repeatedly, we see that there exist $\alpha_k, \beta_k\in \R$ such that
$$
x_n + \alpha_k \leq u(x,t) \leq x_n+ \beta_k \quad \hbox{in}~Q_{r_k}(x_0,t_0),
$$
where $\beta_k - \alpha_k \leq (3 \bar c)^k (4\delta)r_k $, provided $(3\bar c)^k  (4\delta) \leq  \delta_0 $. 
Let $\bar k \geq 1$ and $\gamma \in (0,1)$ be such that $r_{\bar k+1} \leq  2\delta/\delta_0 < r_{\bar k}$ and  $\bar c=3^{-\gamma}$.  Then
$$
\underset{Q_{r_k}(x_0,t_0)} \osc \ u_\delta \leq \frac{\beta_k-\alpha_k}{\delta}\leq  4 r (3^{-\gamma k}) \quad \hbox{for all}~k=0, \hdots, \bar k.$$
If $(x,t) \in Q_r(x_0,t_0)$ satisfies $r_{k+1} \leq d_p ((x,t), (x_0,t_0))< r_k$ for some $k$ as above, then 
\begin{align*}
|u_\delta (x,t) - u_\delta (x_0, t_0)| 
\leq \underset{Q_{r_k}(x_0,t_0)} \osc \ u_\delta  
\leq 4 r^{1-\gamma} 3^\gamma  d_p((x,t), (x_0,t_0))^\gamma.
\end{align*}
\end{proof}

\section{Linear transmission problems} \label{sec:transmission}

We establish the regularity of viscosity solutions to linear transmission problems with flat interfaces that will be fundamental in Sections~\ref{sec:improvement} and \ref{sec:hodograph}. 
Consider the problem:
\begin{equation}
\begin{cases} \label{limitpb}
a_+v_t - a_{ij}(x,t) D_{ij} v = f_+(x,t) &  \hbox{in}~Q_{1}^+,\\
a_-v_t - a_{ij}(x,t) D_{ij} v = f_-(x,t) &  \hbox{in}~Q_{1}^-,\\
 v_n^+ = v_n^- & \hbox{on}~Q_{1}\cap \{x_n=0\},
\end{cases}
\end{equation}
where $v_n^\pm$ is the normal derivative of $v|_{\overline{Q}_1^\pm}$ and $a_{ij}(x,t) D_{ij} v := \tr(A(x,t)D^2v)$. Let $\sym$ be the set of $n\times n$ symmetric matrices. We assume that $A(x,t) \in \sym$ for all $(x,t)\in Q_1$ and  there are constants $0<\underline{a} \leq \bar{a}$ such that
\begin{equation} \label{ellipt}
\underline{a} |\xi|^2 \leq a_{ij}(x,t) \xi_i \xi_j \leq \bar{a} |\xi|^2 \quad \hbox{for all}~(x,t)\in Q_1, \ \xi\in \Rn.
\end{equation}
We will also assume that
\begin{equation} \label{rhs}
a_{ij}, f_\pm\in C(Q_1) \quad \hbox{for}~ i, j=1,\hdots, n.
\end{equation}

\begin{defn} 
We call $P(x,t)$ a \textit{piecewise} quadratic polynomial if 
\begin{align*}
P(x,t) &= \tfrac{1}{2} x^T A(x_n) x + b \cdot x + c t + d,\\
A(x_n) &= A^+ \chi_{\{x_n\geq 0\}} + A^- \chi_{\{x_n<0\}},
\end{align*}
where $A^\pm\in \sym$ satisfies that $A_{ij}^+=A_{ij}^-$ for all $1\leq i, j\leq n$, except possibly for $i=j=n$, $b$ is a vector in $\R^{n}$, and $c$ and $d$ are real numbers. 
\end{defn}

\begin{defn}\label{def:viscflat}
We say that $v\in C(Q)$ is a \textit{viscosity subsolution} to \eqref{limitpb}--\eqref{rhs} if whenever a piecewise quadratic polynomial $P$ touches $v$ from above at $(x_0,t_0)\in Q_1$, the following holds:
\begin{enumerate}[$(i)$]
\item If $(x_0,t_0)\in Q_1^+$, then $a_+c - a_{ij}(x_0,t_0) A_{ij}^+ \leq f_+(x_0,t_0)$. 
\item  If $(x_0,t_0)\in Q_1^-$, then $a_-c - a_{ij}(x_0,t_0) A_{ij}^- \leq f_-(x_0,t_0)$.
\item  If $(x_0,t_0)\in Q_1\cap\{x_n=0\}$, then 
\begin{align*} 
\min\big (a_+c - a_{ij}(x_0,t_0) A_{ij}^+ - f_+(x_0,t_0), \
 a_- c - a_{ij}(x_0,t_0) A_{ij}^-  - f_-(x_0,t_0)\big) \leq 0.
\end{align*}
\end{enumerate}
We say that $v$ is a \textit{viscosity supersolution} if whenever a piecewise quadratic polynomial $P$ touches $v$ from below at $(x_0,t_0)\in Q_1$, then $(i)$ and $(ii)$ hold reversing the inequalities, and

\begin{enumerate}[$(i)$]
\item[$(iii)$]  If $(x_0,t_0)\in Q_1\cap\{x_n=0\}$, then 
\begin{align*} 
\max\big (a_+c - a_{ij}(x_0,t_0) A_{ij}^+ - f_+(x_0,t_0), \ a_- c - a_{ij}(x_0,t_0) A_{ij}^-  - f_-(x_0,t_0)\big) \geq 0.
\end{align*}
\end{enumerate}
We call $v$ a \textit{viscosity solution} if it is both a viscosity subsolution and a viscosity supersolution.
\end{defn}

\begin{rem} 
We obtain an equivalent definition of viscosity solution by replacing $P$ with piecewise $C^2$ functions up to $\{x_n=0\}$ that are differentiable on $Q_1$.
In particular, observe that test functions already satisfy the transmission condition given in \eqref{limitpb}. In order to get meaningful information on the transmission interface, we require a stronger condition on $\{x_n=0\}$, i.e., $v$ must be a subsolution (or a supersolution) to at least one of the equations governing each side. This seems to be a natural definition, given the jump discontinuity of the global operator across this plane, and it is consistent with the following:  
\end{rem}

\begin{lem}
If $v\in C^2(Q_1^+\cup Q_1^-)\cap C^1(Q_1)$ is a classical solution of \eqref{limitpb}--\eqref{rhs}, then $v$ is also a viscosity solution in the sense of Definition~\ref{def:viscflat}. 
\end{lem}

\begin{proof}
Assume that $v$ satisfies \eqref{limitpb} in the classical sense. We only need to check that $v$ satisfies $(iii)$ in Definition~\ref{def:viscflat}. Suppose that $P$ is a piecewise quadratic polynomial that touches $v$ from below at $(x_0, t_0)\in Q_1 \cap \{x_n=0\}$. We have
$$
P(x_0,t_0)= v(x_0,t_0) \quad \hbox{and} \quad P(x,t)\leq v(x,t)~\hbox{for all}~(x,t)\in Q_r(x_0,t_0).
$$

Given $\vep\in (0,r)$ small, let $P_\vep = P + \vep|x_n|$. Define 
$$
c_\vep = \max_{Q_\vep(x_0,t_0)} ( P_\vep -v) =  P_\vep (x_\vep, t_\vep) - v(x_\vep, t_\vep).
$$
Then $P_\vep-c_\vep$ touches $v$ from below at $(x_\vep, t_\vep)\in Q_\vep(x_0,t_0)$. Since $v$ is differentiable, it follows that $(x_\vep, t_\vep) \notin \{x_n=0\}$. Without loss of generality, assume that $(x_\vep, t_\vep) \in Q_1^+$. Using that $v\in C^2(Q_1^+)$, we have $ P_t = \partial_t P_\vep \geq v_t$ and $D_{ij} P^+= D_{ij}P^+_\vep \leq D_{ij} v$ at the point $(x_\vep, t_\vep)$, so
$$
a_+P_t - a_{ij}(x_\vep, t_\vep) D_{ij}P^+ \geq  a_+ v_t - a_{ij}(x_\vep, t_\vep) D_{ij}v^+= f_+(x_\vep,t_\vep),
$$
since $v$ is a classical solution. Moreover, $(x_\vep, t_\vep)\to(x_0,t_0)$ as $\vep\to 0$. Passing to the limit in the above inequality and using the continuity assumptions given in \eqref{rhs}, we conclude that
\begin{align*} 
\max\big (a_+P_t - a_{ij}(x_0,t_0) D_{ij}P^+ - f_+(x_0,t_0), \ a_- P_t - a_{ij}(x_0,t_0) D_{ij}P^-  - f_-(x_0,t_0)\big) \geq 0.
\end{align*}
\end{proof}

We recall the definition of the \textit{Pucci's extremal operators} in \cite{CC}. Given $0<\lambda \leq \Lambda$, denote by $\A_{\lambda,\Lambda}$ the set of matrices in $\sym$ whose eigenvalues belong to $[\lambda, \Lambda]$.  For $M\in \sym$, define
\begin{align*}
\M_{\lambda,\Lambda}^-(M) & := \lambda \sum_{e_i >0} e_i + \Lambda \sum_{e_i<0} e_i = \inf_{ A \in \A_{\lambda, \Lambda}} \tr(AM),  \\
\M_{\lambda,\Lambda}^+(M) &:= \Lambda \sum_{e_i >0} e_i + \lambda \sum_{e_i<0} e_i= \sup_{A\in \A_{\lambda, \Lambda}} \tr(AM),
\end{align*}
where $\{e_i\}_{i=1}^n$ are the eigenvalues of $M$.

\begin{lem}
Let $\lambda= \underline{a} / \max(a_+, a_-)$ and $\Lambda= \bar{a}/\min(a_+,a_-)$. Define the functions $\bar f= \max(f_+/a_+,f_-/a_-)$ and $\underline{f}=\min(f_+/a_+,f_-/a_-)$.
If $v\in C(Q)$ is a viscosity solution to \eqref{limitpb}--\eqref{rhs}, then 
$$
\underline{f} + \M_{\lambda,\Lambda}^-(D^2 v)  \leq   v_t \leq \M_{\lambda,\Lambda}^+(D^2 v) +  \bar{f} \quad \hbox{in}~Q_1, 
$$
in the usual viscosity sense.
\end{lem}

\begin{proof}
It is enough to prove the second inequality:
$$v_t \leq \M_{\lambda,\Lambda}^+(D^2 v) + \bar f \quad \hbox{in} ~Q_1,
$$ 
in the viscosity sense, i.e., we need to show that if a quadratic polynomial, 
$$P(x,t)=\tfrac{1}{2} x^T M x + b \cdot x + c t + d,$$
touches $v$ from above at $(x_0,t_0)\in Q_1$, then
$$
c \leq  \M_{\lambda,\Lambda}^+(M) + \bar f(x_0,t_0). 
$$
 Indeed, suppose that $P$ touches $v$ from above at $(x_0,t_0)\in Q_1$. Then
$$
a_+ c  - a_{ij}(x_0,t_0) M_{ij} \leq  f_+(x_0,t_0) \quad \hbox{or} \quad a_- c  - a_{ij}(x_0,t_0) M_{ij} \leq  f_-(x_0,t_0)
 $$
since $v$ is a viscosity subsolution to \eqref{limitpb}. We get that
$$
c \leq \sup_{ A \in \A_{\lambda, \Lambda} } \tr(A M) + \tfrac{1}{a_\pm} f_\pm (x_0,t_0) \leq  \M_{\lambda,\Lambda}^+(M) + \bar f(x_0,t_0). 
$$
\end{proof}

As a consequence of the previous lemma, we have that viscosity solutions to \eqref{limitpb}--\eqref{rhs} satisfy the maximum principle given in \cite[Theorem~3.14]{W1} and the interior H\"{o}lder estimates in \cite[Theorem~4.19]{W1}. More precisely:

\begin{thm} \label{mp-holder}
Let $f_\pm\in C({Q_1})\cap L^\infty(Q_1)$. Assume that $v\in C({Q_1})$ is a viscosity solution to \eqref{limitpb}--\eqref{rhs} with $v=g \in C(\partial_p Q_1)$. Then $v\in C^{0,\beta}_{loc}(Q_1)$ and
\begin{equation}
\|v\|_{L^\infty(Q_1)} + [v]_{C^{0,\beta}(\overline{Q_{1/2}})} \leq C \Big( \sup_{\partial_p Q_1} |g| + \|f_+\|_{L^{n+1}(Q_1)} + \|f_-\|_{L^{n+1}(Q_1)}\Big),
\end{equation}
where $\beta\in (0,1)$ and $C>0$ depend on $n$, $\underline a$, $\bar a$, $a_+$, and $a_-$. 
\end{thm}

\subsection{Equations with constant coefficients and right-hand side}

To establish the regularity of viscosity solutions to \eqref{limitpb}, we consider first the model problem:

\begin{thm} \label{thm:limitpb}
 Let $a_+,a_->0$ and $ \mu_+, \mu_-\in \R$. Let $v$ be a viscosity solution to 
 \begin{equation}\label{pbv}
 \begin{cases}
 a_+ v_t - \Delta v = \mu_+ & \hbox{in}~Q_1^+,\\
  a_- v_t - \Delta v = \mu_- & \hbox{in}~Q_1^-,\\
  v_n^+ = v_n^- & \hbox{on}~Q_1\cap \{x_n=0\},
  \end{cases}
 \end{equation}
such that $|v|\leq 1$ in $Q_1$.
Then $v(x,t)$ is locally smooth in $x'$ and $t$, and $v\in C_{loc}^{1,1}(Q_1)$. In particular, for any $(x_0,t_0) \in Q_{1/2}$, there exists a linear polynomial $l(x)= b\cdot (x-x_0) + c$, with $|b|+|c| \leq C_0$, such that if $Q_{2r} (x_0,t_0)\subset Q_{1/2}$, then
$$
\sup_{Q_r(x_0,t_0)} |v-l| \leq C_0 r^2 \quad 
$$
for some $C_0>0$ depending on $n$, $a_+$, $a_-$, $\mu_+$, and $\mu_-$.
\end{thm}

To prove the theorem, we will approximate $v$ with functions that satisfy a uniformly parabolic equation with smooth coefficients. We need the following lemma.

\begin{lem} \label{lem:regsmooth}
Let $a: \R \to \R$ be a smooth function such that $0<a_-\leq a(s) \leq a_+$ for all $s\in \R$. 
Let $f: \R \to \R$ be a smooth bounded function. 
If $w$ is a bounded viscosity solution to
\begin{equation}\label{regeq}
a(x_n) w_t - \Delta w= f(x_n) \quad \hbox{in}~Q_1,
\end{equation}
then $w$ is locally smooth, and 
$$
\| w \|_{C^{1,1}(\overline{Q_{1/2}})} \leq C \big( \| w\|_{L^\infty(Q_1)} + \|f\|_{L^\infty(\R)}\big)
$$
for some constant $C>0$ depending on $n$, $a_+$, and $a_-$.
\end{lem}

\begin{proof}
Since $a(x_n)$ and $f(x_n)$ are smooth and uniformly bounded, it is clear that $w$ is locally smooth. Hence, we only need to prove the $C^{1,1}$ estimate. It is enough to show that
\begin{equation} \label{c11est}
\|D^2 w\|_{L^\infty(Q_{3/4})} \leq C \big( \| w\|_{L^\infty(Q_1)} + \|f\|_{L^\infty(\R)}\big)
\end{equation}
for some constant $C>0$ depending on $n$, $a_+$, and $a_-$.
Note that any derivative of $w$ in the $x'$-direction and any time derivative satisfies \eqref{regeq} with zero right-hand side. Hence, carrying out the same procedure as in \cite[Section~5.3]{CC}, we get that
\begin{equation} \label{est:deriv}
    \|D_t^l D_x^m w\|_{L^\infty(Q_{7/8})} \leq C \|w\|_{L^\infty(Q_1)}
\end{equation}
for any integer $l\geq 0$ and multi-index $m=(m_1, \hdots,m_n)$ with $m_n= 0$, where $C>0$ is some constant depending on $n$, $a_+$, and $a_-$. To control the derivatives in the $x_n$-direction, we rewrite the equation in the following way:
\begin{equation} \label{eq:nn}
D_{nn} w(x,t) = a(x_n) w_t (x,t)- \sum_{i=1}^{n-1} D_{ii} w(x,t) + f(x_n) =: \tilde f(x,t).
\end{equation}
By the previous estimate, we get
\begin{equation} \label{est:nn}
    \| D_{nn} w \|_{L^\infty(Q_{7/8})} \leq  (a_-+ n-1)C \|w\|_{L^\infty(Q_1)} +\|f\|_{L^\infty(\R)}.
\end{equation}
Differentiating \eqref{eq:nn} with respect to $x_j$, $j<n$, and integrating with respect to $x_n$, we have 
$$
 D_{jn}w(x',x_n,t)  = \int_c^{x_n} D_j \tilde f(x', y, t) \, dy +  D_{jn}w(x',c,t)
$$
for some $c\in (-7/8, -3/4)$ and any $(x,t)\in Q_{3/4}$. 
Then averaging in $c$ and using \eqref{est:deriv} we get
\begin{equation} \label{est:jn}
    \| D_{jn} w \|_{L^\infty(Q_{3/4})} \leq  C \|w\|_{L^\infty(Q_1)}
\end{equation}
for some constant $ C>0$ depending on $n$, $a_+$, and $a_-$.

Therefore, combining the estimates in \eqref{est:nn} and \eqref{est:jn}, we conclude \eqref{c11est}.
\end{proof}

\begin{proof}[Proof of Theorem~\ref{thm:limitpb}]
Without loss of generality, assume that $0<a_-< a_+$ and $\mu_-\leq \mu_+$.
Let $a^j: \R\to\R$ be a smooth function satisfying $a^j(s)=a_-$ if $s<-\tfrac{1}{j}$, $a^j(s)=a_+$ if $s>\tfrac{1}{j}$, and $a_-\leq  a^j(s)\leq a_+$ for all $s\in \R$. Similarly, let $f^j : \R \to \R$ be a smooth function satisfying $f^j(s)=\mu_-$ if $s<-\tfrac{1}{j}$, $f^j(s)=\mu_+$ if $s>\tfrac{1}{j}$, and $\mu_- \leq  f^j(s)\leq \mu_+$ for all $s\in \R$.
Let $w^j$ be the viscosity solution to 
$$
\begin{cases}
a^j(x_n) w^j_t - \Delta w^j=f^j(x_n) & \hbox{in}~Q_{3/4},\\
w^j = v & \hbox{on}~\partial_p Q_{3/4}.
\end{cases}
$$
By Lemma~\ref{lem:regsmooth}, $w^j$ is locally smooth and for any $r\in (0, 3/4)$, we have
\begin{equation} \label{est:c11wj}
    \| w^j \|_{C^{1,1}(\overline{Q_r})} \leq C(r)
\end{equation}
for some constant $C(r)>0$ independent of $j$.
By compactness, up to a subsequence, $w^j \to w$ as $j\to\infty$ in $C_{loc}^{1,\beta}$ for any $\beta \in (0,1)$. Furthermore, $w\in C_{loc}^{1,1}(Q_{3/4})\cap C(\overline{Q_{3/4}})$ and  
\begin{equation*}
\begin{cases}
a_+w_t - \Delta w = \mu_+ &  \hbox{in}~Q_{3/4}^+,\\
a_-w_t - \Delta w= \mu_- &  \hbox{in}~Q_{3/4}^-,\\
 w=v & \hbox{on}~\partial_p Q_{3/4},\\
\end{cases}
\end{equation*}
in the classical sense. 
By the maximum principle (Theorem~\ref{mp-holder}), applied to $w-v$, it follows that $w=v$ in $Q_{3/4}$. 
 Hence, $v\in C^{1,1}_{loc}(Q_{3/4})$ and by \eqref{est:c11wj}, we get
  $$\|v\|_{C^{1,1}(\overline{Q_{1/2}})} \leq C_0,$$
 where $C_0>0$ is a constant depending on $n$, $a_+$, $a_-$, $\mu_+$, and $\mu_-$.
  In particular, for any $(x_0,t_0)\in Q_{1/2}$, there exists a linear polynomial, $l(x)= \nabla v(x_0,t_0) \cdot (x-x_0) + v(x_0,t_0)$, with $|\nabla v(x_0,t_0)|+|v(x_0,t_0)| \leq C_0$, such that
$$
\sup_{Q_r(x_0,t_0)} |v-l| \leq C_0 r^2 
$$
for $r>0$ small enough so that $Q_{2r}(x_0,t_0)\subset Q_{1/2}$.
\end{proof}

\subsection{Equations with variable coefficients and right-hand side}

We prove $C^{2,\alpha}$ estimates up to the flat interface for viscosity solutions to \eqref{limitpb}--\eqref{rhs}. 
 
 \begin{thm} \label{thm:bootstrap}
 Fix $\alpha\in (0,1)$. 
 Assume that $a_{ij}, f_\pm \in C_{loc}^{0,\alpha}(Q_1)$ for all $i,j=1,\hdots,n$.
 If $w$ is a bounded viscosity solution to \eqref{limitpb}--\eqref{ellipt}, then $$w\in C^{2,\alpha}(\overline{Q_{1/2}^+})\cap C^{2,\alpha}(\overline{Q_{1/2}^-})$$  
 and the following estimate holds:
 $$
 \|w\|_{C^{2,\alpha}(\overline{Q_{1/2}^+})} +  \|w\|_{C^{2,\alpha}(\overline{Q_{1/2}^-})} \leq C \big( \|w\|_{L^\infty(Q_1)} + \|f_+\|_{C^{0,\alpha}(\overline{Q_{3/4}})} + \|f_-\|_{C^{0,\alpha}(\overline{Q_{3/4}})}\big), 
 $$
 where $C>0$ depends on $n$, $a_+$, $a_-$, and $\|a_{ij}\|_{C^{0,\alpha}(\overline{Q_{3/4}})}$.
 \end{thm}

Theorem~\ref{thm:bootstrap} follows from the next pointwise $C^{2,\alpha}$ estimate on the interface, using a standard argument of patching the interior and boundary estimates; see \cite[Proposition~2.4]{MS}. 

\begin{thm} \label{thm:pointwiseest}
Fix $\alpha\in (0,1)$. Assume that $a_{ij} \in C^{\alpha}(0,0)$, $a_{ij}(0,0)=\delta_{ij}$, and $f_\pm\in C^\alpha(0,0)$. If $w$ is a viscosity solution to \eqref{limitpb}--\eqref{rhs} with $|w|\leq 1$, then 
$$w|_{\overline{Q_{1/2}^\pm}} \in C^{2,\alpha}(0,0),$$
i.e., there is a piecewise quadratic polynomial, $P(x,t)= \frac{1}{2} x^T A(x_n) x + b\cdot x + ct + d$, such that
$$
\|w-P	\|_{L^\infty(Q_r)} \leq C r^{2+\alpha} \quad \hbox{for all}~ r\leq  r_0,
$$
and some $ r_0 \in (0,1/4)$ and $C>0$, where the coefficients of $P$ satisfy:
$$
 a_+ c = \tr(A^+) + f_+(0,0),\quad a_- c = \tr(A^-) + f_-(0,0).
$$
Moreover, there is $C_0>0$ depending on $n$, $a_+$, $a_-$, and $[a_{ij}]_{C^\alpha(0,0)}$ such that
$$
\|A^+\|+\|A^-\|+|b|+|c|+|d|\leq C_0 \big([f_+]_{C^\alpha(0,0)} + [f_-]_{C^\alpha(0,0)} \big).
$$
\end{thm}

To prove the theorem, we need a couple of lemmas.

\begin{lem} \label{lem1}
 Fix $\vep \in (0,1/300)$.  Let $w$ be a viscosity solution to \eqref{limitpb}--\eqref{rhs} such that $|w|\leq 1$ in $Q_1$. If for all $i,j=1,\hdots,n$ we have
$$
 \|a_{ij}-\delta_{ij}\|_{L^{n+1}(Q_1)} + \|f_\pm - f_\pm(0,0)\|_{L^{n+1}(Q_1)}  \leq \vep,
$$
then the classical solution $v$ to \eqref{pbv} in $Q_{3/4}$ with $\mu_\pm=f_\pm(0,0)$ and $v=w$ on $\partial_p Q_{3/4}$ satisfies
$$
\|w - v\|_{L^\infty(Q_{1/2})}\leq C \vep^{\beta/4}
$$
for some $C>0$, and $\beta \in (0,1)$ as in Theorem~\ref{mp-holder}.
\end{lem}

\begin{proof}
Let $v$ be as in the statement. Then the function $u=w-v$ satisfies
 \begin{equation*}
\begin{cases} 
a_+u_t - a_{ij}(x,t) D_{ij} u = g_+(x,t) &  \hbox{in}~Q_{3/4}^+,\\
a_-u_t - a_{ij}(x,t) D_{ij} u = g_-(x,t) &  \hbox{in}~Q_{3/4}^-,\\
 u_n^+ = u_n^- & \hbox{on}~Q_{3/4}\cap \{x_n=0\},
\end{cases}
\end{equation*}
in the viscosity sense, where $g_\pm(x,t) = (a_{ij}(x,t)-\delta_{ij})D_{ij}v + f_\pm(x,t)- f_\pm(0,0)$. Fix some $r\in (0,1/4)$. By the maximum principle and the interior $C^{0,\beta}$ estimate  in Theorem~\ref{mp-holder}, and the interior $C^{1,1}$ estimate of $v$ in Theorem~\ref{thm:limitpb}, we have
\begin{align*}
\|u\|_{L^\infty(Q_{3/4-r})} & \leq C \big( \|u\|_{L^\infty(\partial_p Q_{3/4-r})} + \|g_+\|_{L^{n+1}(Q_{3/4})}+ \|g_-\|_{L^{n+1}(Q_{3/4})}\big)\\
&\leq C\big( \|u\|_{L^\infty(\partial_p Q_{3/4-r})} +  \max_{i,j} \|a_{ij}-\delta_{ij}\|_{L^{n+1}(Q_1)} \|D^2 v\|_{L^\infty(Q_{3/4-r})} \\
& \qquad \qquad  \|f_+(x,t)- f_+(0,0)\|_{L^{n+1}(Q_1)} + \|f_-(x,t)- f_-(0,0)\|_{L^{n+1}(Q_1)}\big) \\
&\leq C \big ([w]_{C^\beta(\overline{Q_{3/4-r}})}+[v]_{C^\beta(\overline{Q_{3/4-r}})}\big) r^\beta + C \vep r^{-2}+ C\vep\\
& \leq C(\vep^{\beta/4}+ \vep^{1/2}) \leq C\vep^{\beta/4},
\end{align*}
where the second last inequality follows choosing $r=\vep^{1/4}<1/4$.
\end{proof}

\begin{lem} \label{lem2}
Given $\alpha \in (0,1)$, there exist $\bar \vep \in(0,1/300)$, $\bar r \in (0,1/4)$, and $C_0>0$, such that if $w$ is a viscosity solution to \eqref{limitpb}--\eqref{rhs} with $|w|\leq 1$ in $Q_1$ and
$$
 \|a_{ij}-\delta_{ij}\|_{L^{n+1}(Q_1)} +  \|f_\pm - f_\pm(0,0)\|_{L^{n+1}(Q_1)}  \leq \bar \vep,
$$
then there is a piecewise quadratic polynomial, $P(x,t)= \frac{1}{2} x^T A(x_n) x + b\cdot x + ct + d$, such that
$$
\|w-P\|_{L^\infty(Q_{\bar r})} \leq \bar r^{2+\alpha}.
$$
Moreover, $\|A^+\|+\|A^-\|+|b|+|c|+|d|\leq C_0$, and
\begin{align*}
 a_+ c  = \tr(A^+) + f_+(0,0), \quad  a_- c = \tr(A^-) + f_-(0,0).
\end{align*}

\end{lem}

\begin{proof}
Fix $\bar \vep$ and  $\bar r$ to be determined. 
Let $v$ be as in Lemma~\ref{lem1}. By Theorem~\ref{thm:limitpb}, we know that $v$ is piecewise smooth and $C^{1,1}$ at $(0,0)$.
Define $P$ as in the statement, where
$$
A^\pm = \lim_{Q_{1/2}^\pm\ni(x,t)\to (0,0)} D^2 v(x,t), \ b=\nabla v(0,0), \ c= v_t(0,0), \ d= v(0,0).
$$
From the regularity of $v$, we know that the coefficients satisfy the conditions in the statement. Furthermore, using Lemma~\ref{lem1}, we have
$$
\|w-P\|_{L^\infty(Q_{\bar r})} \leq \|w-v\|_{L^\infty(Q_{\bar r})} + \|v-P\|_{L^\infty(Q_{\bar r})} \leq C \bar \vep^{\beta/4} + C \bar r^3.
$$
First, choose $\bar r$ small enough so that 
$C \bar r^3 \leq \bar{r}^{2+\alpha}/2.$
Then choose $\bar \vep$ small enough so that $C \vep^{\beta/4}\leq \bar{r}^{2+\alpha}/2$. We obtain the desired result. 
\end{proof}

Finally, we prove the main theorem of this section.

\begin{proof}[Proof of Theorem~\ref{thm:pointwiseest}]
Let $\bar\vep$ and $\bar r$ be as in Lemma~\ref{lem2}. Fix $\vep \leq \bar\vep$ sufficiently small.
We may normalize the problem so that $|w|\leq 1$ in $Q_1$ and
$
[a_{ij}]_{C^\alpha(0,0)} + [f_+]_{C^\alpha(0,0)} +  [f_-]_{C^\alpha(0,0)} \leq \vep.
$
Indeed, consider the rescaled function
$$
\tilde w(x,t) = \frac{\rho^{-2} w(\rho x, \rho^2 t)}{K} \quad\hbox{for}~ (x,t)\in Q_1,
$$
where $K= \rho^{-2}\|w\|_{L^\infty(Q_1)}+\vep^{-1}( [f_+]_{C^\alpha(0,0)}+ [f_-]_{C^\alpha(0,0)})$.
Then $\tilde w$ satisfies \eqref{limitpb} with $|\tilde w|\leq 1$ in $Q_1$,  $\tilde a_{ij}(x,t)= a_{ij}(\rho x, \rho^2 t)$, and
$\tilde f_\pm (x,t) = f_\pm(\rho x, \rho^2 t) / K$.
Clearly, 
$$ [\tilde f_+]_{C^\alpha(0,0)} +  [ \tilde f_-]_{C^\alpha(0,0)} \leq \vep
\quad \hbox{and}\quad 
[\tilde a_{ij}]_{C^\alpha(0,0)} =  [a_{ij}]_{C^\alpha(0,0)}   \rho^\alpha \leq \vep,$$
choosing $\rho>0$ sufficiently small.\medskip

It is enough to show the following discrete version of the theorem:\medskip

\textbf{Claim.} \textit{For each $k\geq 1$, there exists a piecewise quadratic polynomial, 
$$
P_k(x,t)= \frac{1}{2} x^T A_k(x_n) x + b_k\cdot x + c_k t + d_k, 
$$
satisfying $a_+c_k = \tr(A_k^+) + f_+(0,0)$  and  $ a_- c_k = \tr(A_k^-) + f_-(0,0)$, such that 
$$
\| w-P_k\|_{L^\infty(Q_{\bar r^k})} \leq  \bar r^{k(2+\alpha)}.
$$
Moreover, there is $C_0>0$ such that
\begin{align} \label{estcoeff}
 &\bar r^{2k}  \|A_{k+1}^+-A_{k}^+\| +\bar r^{2k}  \|A_{k+1}^--A_{k}^-\|
+ \bar r^{k} |b_{k+1}-b_{k}| \\ \nonumber
 &\qquad +  \bar r^{2k}  |c_{k+1}-c_{k}|+|d_{k+1}-d_{k}|\leq C_0 \bar r^{k(2+\alpha)}.
\end{align}
}
\medskip

We will prove the claim by induction. For $k=1$, by Lemma~\ref{lem2}, there is a piecewise quadratic polynomial $P_1$ satisfying the above conditions. Assume the claim holds for some $k>1$. Consider the function
\begin{equation}\label{wk}
w_k (x,t) = \frac{w(\bar r^k x, \bar r^{2k} t)- P_k (\bar r^k x, \bar r^{2k} t)}{\bar r^{k(2+\alpha)}} \quad \hbox{for}~ (x,t)\in Q_1.
\end{equation}
Then $|w_k|\leq 1$ in $Q_1$ and $w_k$ satisfies
\begin{equation*}
\begin{cases} 
a_+\partial_t w_k - (a_k)_{ij}(x,t) D_{ij} w_k = (f_+)_k(x,t) &  \hbox{in}~Q_{1}^+,\\
a_-\partial_t  w_k - (a_k)_{ij}(x,t) D_{ij} w_k = (f_-)_k(x,t) &  \hbox{in}~Q_{1}^-,\\
 (w_k)_n^+ = (w_k)_n^- & \hbox{on}~Q_{1}\cap \{x_n=0\},
\end{cases}
\end{equation*}
where 
\begin{align*}
(a_k)_{ij} (x,t) & = a_{ij}(\bar r^k x, \bar r^{2k} t)\\
(f_+)_k (x,t) & = \bar r^{-k\alpha} \big( f_+(\bar r^k x, \bar r^{2k} t)- f_+(0,0)\big) + \bar r^{-k\alpha} \big(\delta_{ij} - (a_k)_{ij}(x,t)\big) (A_k^+)_{ij},\\
(f_-)_k (x,t) & =  \bar r^{-k\alpha} \big( f_-(\bar r^k x, \bar r^{2k} t)- f_-(0,0)\big) + \bar r^{-k\alpha} \big(\delta_{ij} - (a_k)_{ij}(x,t)\big) (A_k^-)_{ij}.
\end{align*}

We have that
$$
\|(a_k)_{ij} - \delta_{ij}\|_{L^\infty(Q_1)} \leq [a_{ij}]_{C^\alpha(0,0)} \bar r^{k\alpha} \leq \bar \vep. 
$$
Also, $(f_\pm)_k(0,0)=0$ and
$$
\|(f_\pm)_k\|_{L^\infty(Q_1)} \leq [f_\pm]_{C^\alpha(0,0)} + [a_{ij}]_{C^\alpha(0,0)} \|A_k^\pm\| \leq (1+C_0)\vep \leq \bar \vep,
$$
taking $\vep$ small enough. 
Hence, $w_k$ satisfies the assumptions of Lemma~\ref{lem2}, so there exists some piecewise quadratic polynomial, $\bar P(x,t) =  \frac{1}{2} x^T \bar A(x_n) x + \bar b\cdot x + \bar ct + \bar d$, satisfying  $a_+\bar c  = \tr(\bar A^+)$ and  $a_-\bar c = \tr(\bar A^-)$, such that
$$
\| w_k - \bar P\|_{L^\infty(Q_{\bar r})}\leq \bar r^{2+\alpha}.
$$
Moreover, $\|\bar A^+\|+\|\bar A^-\|+|\bar b|+|\bar c|+| \bar d|\leq C_0$.
 By \eqref{wk}, the above estimate is equivalent to
\begin{equation*}
 \sup_{Q_{\bar r}} | w(\bar r^k x, \bar r^{2k} t) - P_k( \bar r^k x, \bar r^{2k} t) -\bar r^{k(2+\alpha)} \bar P(x,t)| \leq \bar r^{(k+1)(2+\alpha)}.
 \end{equation*}
Taking 
$P_{k+1}(x,t) = P_k(x,t) + \bar r^{k(2+\alpha)} \bar P(x/\bar r^k, t/ \bar r^{2k}),
$ 
it follows that
$$
\|w-P_{k+1}\|_{L^\infty(Q_{\bar r^{k+1}})}\leq \bar r^{(k+1)(2+\alpha)}.
$$
By definition of $P_{k+1}$, we have
\begin{align*}
A_{k+1}^\pm &= A_k^\pm + \bar r^{k\alpha} \bar A,\\
b_{k+1} &= b_k + \bar r^{k(1+\alpha)} \bar b,\\
c_{k+1} &= c_k + \bar r^{k\alpha} \bar c,\\
d_{k+1} &= d_k + \bar r^{k(2+\alpha)} \bar d.
\end{align*}
Then it is clear that the coefficients satisfy the estimate in \eqref{estcoeff}. Finally, 
\begin{align*}
a_\pm c_{k+1} & = a_\pm c_k +  \bar r^{k\alpha} a_\pm \bar c
 = \tr( A_k^\pm) + f_\pm(0,0) +  \bar r^{k\alpha} \tr(\bar A) 
= \tr(A_{k+1}^\pm)  + f_\pm(0,0) .
\end{align*}
Therefore, we conclude the desired result.
\end{proof}

\section{Improvement of flatness} \label{sec:improvement}

The following theorem says that if the distance of $u$ to $x_n$ is at most $\delta$ in the unit cylinder $Q_1$, then in a smaller cylinder $Q_r$, the distance of $u$ to a possibly rotated plane, close to $x_n$, is at most $\tfrac{\delta}{2}r$. An iteration of this result will give the $C^{1,\alpha}$ regularity of the free boundary. 

\begin{thm}[Improvement of flatness] \label{improvflat}
Let $u$ be a viscosity solution to \eqref{FBP} in $Q_1$, with $(0,0)\in F(u)$, satisfying the flatness condition:
\begin{equation} \label{eq:flatness}
\sup_{Q_1}|u(x,t)-x_n|\leq \delta.
\end{equation}
There are constants $\bar r \in (0,1/2)$ and $C_0>0$,  depending on $n$, $a_+,$ and $a_-$, such that for all $r \in (0, \bar r]$, there is $\bar \delta=\bar \delta( r)>0$ such that: if $\delta \in (0, \bar \delta]$, then there exists some vector $\bar \nu \in \Rn$, with $| \bar\nu - e_n|\leq  C_0 \delta$, for which
$$
\sup_{Q_{r}} |u(x,t)- \bar \nu \cdot x | \leq \frac{\delta}{2} {r}.
$$
\end{thm}

\begin{proof}
We divide the proof into three steps:\medskip

\textbf{Step 1.} Assume by contradiction that the statement is false. 
Fix $\bar r\in(0,1/2)$ and $C_0>0$ to be chosen later. Then there exists some $r\in (0, \bar r]$, a sequence $\delta_k\to 0$ as $k\to \infty$, and viscosity solutions $u_k$ to \eqref{FBP} in $Q_1$ with $(0,0) \in F(u_k)$ for each $k\geq $1, such that
\begin{equation} \label{flatk}
\sup_{Q_1}|u_k(x,t) - x_n | \leq \delta_k,
\end{equation}
but the conclusion of the theorem does not hold, i.e.,
\begin{equation} \label{contk}
\sup_{Q_{ r}} |u_k(x,t) - \bar \nu \cdot x| > \frac{\delta_k}{2} r,
\end{equation}
for any $\bar \nu\in \R^{n}$ with $|\bar \nu-e_n|\leq C_0\delta_k$.
Let 
\begin{equation}\label{defvk}
v_k = \frac{u_k - x_n}{\delta_k}.
\end{equation}
By \eqref{flatk} we have that $|v_k|\leq 1$ in $Q_1$. By Corollary~\ref{holder}, if $k$ is large enough so that $\delta_k \leq  \delta_0/4$, then for any $(x,t)\in Q_{1/2}$ and $(y,s) \in Q_{1/2}(x,t)$ with $d \equiv d_p((x,t),(y,s)) \geq 2 \delta_k / \delta_0$, we get
\begin{equation*} 
|v_k(x,t)- v_k(y,s)| \leq C d^{\gamma}.
\end{equation*}
Furthermore, $F(u_k)\subseteq Q_1\cap \{|x_n|\leq \delta_k\}$, and thus, 
\begin{equation} \label{eq:fbconv}
d_{\h} ( F(u_k), Q_1\cap \{x_n=0\} ) \to 0 \quad \hbox{as}~\delta_k \to 0, 
\end{equation}
where $d_{\h}$ is the Hausdorff distance, defined in Section~\ref{sec:appendix}. By a simple modification of Lemma~\ref{lem:appendix}, there exists a function $v\in C^{0,\gamma}(Q_{1/2})$ such that $\|v\|_{C^{\gamma}(Q_{1/2})}\leq C$, and
\begin{equation} \label{eq:graphs}
d_{\h} ( G(v_k), G(v)) \to 0 \quad \hbox{as}~\delta_k \to 0,
\end{equation}
where $G(v_k)$ and $G(v)$ are the graphs of $v_k$ and $v$ in $Q_{1/2}$, respectively. 
\medskip

\textbf{Step 2.}
Next, we show that $v$ is a viscosity solution to the transmission problem
\begin{equation*} 
\begin{cases}
a_+v_t - \Delta v = 0 &  \hbox{in}~Q_{1/2}^+,\\
a_-v_t - \Delta v = 0 &  \hbox{in}~Q_{1/2}^-,\\
 v_n^+ = v_n^- & \hbox{on}~Q_{1/2}\cap \{x_n=0\},
\end{cases}
\end{equation*}
in the sense of Definition~\ref{def:viscflat}. Assume by contradiction that there is a piecewise quadratic polynomial, $P(x,t) = \tfrac{1}{2} x^T A(x_n) x + b \cdot x + c t + d$ with $A(x_n) = A^+ \chi_{\{x_n\geq 0\}} + A^- \chi_{\{x_n<0\}},$
 that touches $v$ from below at some $(x_0,t_0) \in Q_{1/2}$, but the following holds:
\begin{enumerate}[$(i)$]
\item If $(x_0,t_0)\in Q_{1/2}^+$, then $a_+c - \tr(A^+) < 0$. 
\item  If $(x_0,t_0)\in Q_{1/2}^-$, then $a_-c - \tr(A^-) < 0$.
\item  If $(x_0,t_0)\in Q_{1/2}\cap\{x_n=0\}$, then $\max\big(a_+c - \tr(A^+), \ a_- c - \tr(A^-) \big) < 0.$
\end{enumerate}
We may assume that $P$ touches $v$ strictly from below at $(x_0,t_0)$. Otherwise, we take $P-\eta |x-x_0|^2 - \eta^2 (t_0-t)$ with $\eta>0$ sufficiently small.
By \eqref{eq:graphs}, there exist points $(x_k, t_k) \in Q_{1/2}$ with $(x_k,t_k)\to(x_0,t_0)$ and a sequence of real numbers $d_k\to 0$ as $k\to\infty$, such that
$$
\begin{cases}
P(x_k,t_k) + d_k = v_k (x_k,t_k), \\
P+d_k \leq v_k \ \hbox{in}~Q_\rho(x_k,t_k) \ \hbox{for some}~\rho>0 \ \hbox{small}.
\end{cases}
$$
Let $\tilde P =\delta_k(P+d_k) + x_n.$ It follows from \eqref{defvk} that $\tilde P$ touches $u_k$ from below at $(x_k,t_k)$.  
If $(x_0,t_0) \in Q_{1/2}^+$, then by \eqref{eq:fbconv}, we have $(x_k,t_k) \in Q_{1/2}^+(u_k)$ for $k$ sufficiently large. Hence, by definition of viscosity supersolution, we get
$$
a_+\tilde P_t - \Delta \tilde P = a_+ \delta_k c  - \delta_k \tr(A^+) \geq 0,
$$
and dividing by $\delta_k$, we reach a contradiction with $(i)$.
Hence, $(x_0,t_0)\notin Q_{1/2}^+$. A similar argument shows that $(x_0,t_0)\notin Q_{1/2}^-$. Therefore, $(x_0,t_0)\in Q_{1/2}\cap\{x_n=0\}$. In this case, there exists some $k\geq 1$ such that $(x_k,t_k)\notin F_{1/2}(u_k)$. Hence, we must have that $a_+c - \tr(A^+)\geq 0$ or  $a_-c - \tr(A^-) \geq 0$, which contradicts $(iii)$. 
\medskip

\textbf{Step 3.}
By Theorem~\ref{thm:limitpb}, $v\in C^{1,1}(0,0)$, and there exists a constant $C_0>0$ such that
$$
\sup_{Q_r} |v(x,t)-b \cdot x| \leq C_0 r^2, 
$$
where $|b| \leq C_0$. Notice that $c=0$ since $(0,0)\in F(u)$. Hence, in view of \eqref{eq:graphs}, it follows that
 $$
\sup_{Q_r} |v_k(x,t)-b\cdot x| \leq  2C_0 r^2
$$
for $k$ sufficiently large. From the definition of $v_k$ in \eqref{defvk}, we get
$$
\sup_{Q_r} |u_k(x,t) - x_n - \delta_k b \cdot x | \leq  2C_0 \delta_k r^2.
$$
Let $\nu_k = \delta_k b + e_n$. Then $|\nu_k-e_n|\leq C_0\delta_k$, and choosing $\bar r \in (0,1/2)$ small enough so that $2 C_0   \bar r \leq 1/2$, we see that
$$
\sup_{Q_r} |u_k(x,t) - \nu_k \cdot x| \leq \frac{\delta_k}{2} r,
$$
which is a contradiction with \eqref{contk}.
\end{proof}

\section{$C^{1,\alpha}$ regularity of the free boundary} \label{sec:fbregularity}

Theorem~\ref{thm2} will be a consequence of the following $C^{1,\alpha}$ estimate at points on the free boundary of viscosity solutions to \eqref{FBP} satisfying the flatness assumption.

\begin{thm} \label{thm:improvflat2}
Let $ \bar r \in (0,1/2)$ and $\bar \delta=\bar \delta(\bar r)>0$ be as in Theorem~\ref{improvflat}. 
Let $u$ be a viscosity solution to \eqref{FBP} satisfying \eqref{eq:flatness} for some $\delta \in (0, \bar \delta]$. There exists $\alpha = \alpha(\bar r) \in (0,1)$ such that for any $(x_0,t_0)\in F_{1/2}(u)$, there is $\nu_{x_0,t_0} \in \Rn$, with $|\nu_{x_0,t_0}-e_n|\leq C\delta$, for which
$$
\sup_{Q_r(x_0,t_0)} |u(x,t) - \nu_{x_0,t_0} \cdot (x-x_0)| \leq C \delta r^{1+\alpha} \quad \hbox{for all}~ r\leq \bar r,
$$ 
where $C>0$ depends on $n$, $a_+$, and $a_-$.
\end{thm}

\begin{proof}
For each $k\geq 0$, consider the blow-up sequence of functions
$$
u_k(x,t) = \frac{u( \bar r^k x+x_0, \bar r^{2k} t+t_0)}{\bar r^k} \quad \hbox{for}~ (x,t)\in Q_1.
$$
Then $u_k$ is a viscosity solution of \eqref{FBP}. 
Define  $\alpha\in (0,1)$ such that $ \bar r^{\alpha} = 1/2$.
We will prove by induction that there is a constant $C_0>0$ and vectors $\nu_k \in \Rn$ such that
\begin{align*} \label{claim}
\sup_{Q_1} |u_k(x,t) - \nu_k \cdot x |& \leq \delta  \bar r^{k \alpha},\\
|\nu_{k+1} - \nu_{k}| &\leq C_0 \delta  \bar r^{k \alpha}.
\end{align*}
Indeed, if $k=0$, then $\nu_0=e_n$, and the estimate follows trivially from \eqref{flatness}. 
Assume the claim holds for some $k\geq 1$.  Then
$$
\sup_{Q_1} |u_k(x,t)- (\nu_k-e_n) \cdot x - x_n | \leq \delta_k,
$$
where $\delta_k=\delta \bar r^{k \alpha}.$ By Theorem~\ref{improvflat} applied to the function $u_k(x,t)- (\nu_k-e_n) \cdot x$ (note that $(0,0)$ is on the free boundary), there exists some $\bar\nu_k\in \R^{n}$ with $|\bar\nu_k - e_n|\leq C_0 \delta_k$, for some $C_0>0$, such that
$$
\sup_{Q_{\bar r}} |u_k(x,t)- (\nu_k-e_n) \cdot x - \bar\nu_k \cdot x | \leq \frac{\delta_k}{2} \bar r.
$$
 Then $u_{k+1}(x,t)= \bar r^{-k} u_k ( \bar r^k x ,  \bar r^{2k} t)$ satisfies
$$
\sup_{Q_1} |u_{k+1}(x,t)- \nu_{k+1}\cdot x  | \leq \frac{\delta_k}{2}=\delta \bar r^{(k+1)\alpha} ,
$$
where $\nu_{k+1}=\nu_k+\bar\nu_k-e_n$, so $|\nu_{k+1}-\nu_k|\leq C_0\delta_k=C_0\delta \bar r^{k\alpha}$. Therefore, there exists some $\nu_{x_0,t_0}\in \Rn$ such that $\nu_k \to \nu_{x_0,t_0}$ as $k\to \infty$, and 
\begin{align*}
|\nu_{x_0,t_0} - e_n| \leq \sum_{k=0}^\infty |\nu_{k+1}-\nu_k|\leq \frac{C_0\delta}{1-\bar r^{\bar \alpha}}\leq C\delta.
\end{align*}
Given $r \leq \bar r$, there exists some $k\geq 0$ such that $\bar r^{k+1} < r \leq \bar r^k$. We conclude that:
\begin{align*}
\sup_{Q_r(x_0,t_0)} |u(x,t) - \nu_{x_0,t_0} \cdot (x-x_0)| & 
\leq \bar r^k \sup_{Q_1} |u_k(x,t) - \nu_{x_0,t_0}  \cdot x| \\
& \leq \bar r^k \sup_{Q_1} |u_k(x,t) - \nu_k  \cdot x| + \bar r^k |\nu_{x_0,t_0}- \nu_k|\\
& \leq \delta \bar r^{k(1+\alpha)} + \delta C_0 \bar r^k \sum_{j= k}^\infty \bar r^{j \alpha} \\
& \leq \frac{\delta}{\bar r^{1+\alpha}} \Big(1+\frac{C_0}{1-\bar r^{\alpha}}\Big) \bar r^{(k+1)(1+\alpha)}
\leq C \delta r^{1+\alpha}.
\end{align*}
\end{proof}

\begin{rem} \label{rem:normal}
The estimate in Theorem~\ref{thm:improvflat2} implies that for any $(x_0,t_0)\in F_{1/2}(u)$, we have  $\nabla u(x_0,t_0) = \nu_{x_0,t_0}$ and $u$ is $C^{1,\alpha}$ at that point. In particular, the free boundary condition in \eqref{FBP} is satisfied in the classical sense.
\end{rem}

\begin{cor}[$C^{0,\alpha}$ regularity of $\nu$] \label{cor:normal}
Suppose we are under the assumptions of Theorem~\ref{thm:improvflat2}.
Then for any  $(x_1, t_1), (x_2, t_2) \in F_{1/2}(u)$, with $r=d_p((x_1, t_1), (x_2, t_2))\leq \bar r/3$, we have
$$
|\nu_{x_1,t_1} - \nu_{x_2,t_2}|\leq C\delta r^{\alpha},
$$
for some $C>0$ depending on $n$, $a_+$, and $a_-$.
\end{cor}

\begin{proof}
Let $(x_1, t_1), (x_2, t_2) \in F_{1/2}(u)$ with $r=d_p((x_1, t_1), (x_2, t_2))\leq \bar r/3$. Then $Q_r(x_1,t_1) \subseteq Q_{3r}(x_2,t_2)$, and applying Theorem~\ref{thm:improvflat2} three times, we get
\begin{align*}
|\nu_{x_1,t_1} - \nu_{x_2,t_2}|^2 &= \frac{c_n}{r^{n+2}} \int_{B_r(x_1)}  |\nu_{x_1,t_1}\cdot(x-x_1) - \nu_{x_2,t_2}\cdot(x-x_1)|^2\, dx\\
&\leq \frac{\tilde c_n}{r^2} \ \Big( \sup_{Q_r(x_1,t_1)} |u(x,t)-\nu_{x_1,t_1}\cdot(x-x_1)|\\
&\qquad \qquad \quad +   \sup_{Q_{3r}(x_2,t_2)} |u(x,t)-\nu_{x_2,t_2}\cdot(x-x_2)|\\
& \qquad\qquad \qquad \qquad+  |u(x_1,t_1)-\nu_{x_2,t_2}\cdot(x_1-x_2)|\Big)^2\\
&\leq \frac{\tilde c_n}{r^2}  \big( C\delta r^{1+\alpha} + C\delta r^{1+\alpha} + C\delta r^{1+\alpha}\big)^2
\leq \tilde C \delta^2 r^{2\alpha},
\end{align*}
where we used that  $u(x_1,t_1)=0$, and $\tilde C>0$ depends only on $n$, $a_+$, and $a_-$.
\end{proof}

We now give the proof of our second main theorem:

\begin{proof}[Proof of Theorem~\ref{thm2}]
Let $\bar r\in(0,1/2)$ and $\bar \delta=\bar\delta (\bar r)>0$ be as in Theorem~\ref{improvflat}. Fix $\delta \in (0, \bar \delta]$ to be chosen small. 
First, we show that $F_{1/2}(u)$ is a graph in the $e_n$-direction.
Assume by contradiction that there exist two points $(x_1,t_1), (x_2,t_2) \in F_{1/2} (u)$ such that $x_1'=x_2'$ and $t_1=t_2$, but $(x_2-x_1)\cdot e_n =y_2-y_1 \neq 0$. Without loss of generality, $\rho =2 |y_1-y_2|\leq \bar r$. Then $(x_2,t_2) \in Q_{\rho}(x_1,t_1)$ and, by Theorem~\ref{thm:improvflat2}, we have
$$
|\nu_{x_1,t_1}- e_n| \leq C \delta \quad \hbox{and} \quad |\nu_{x_1,t_1}\cdot (x_2-x_1)| \leq C \delta \rho^{1+\alpha},
$$
since $u(x_2,t_2)=0$. The second inequality implies that $|\nu_{x_1,t_1}\cdot e_n| \leq  C \delta$. Choose $\delta$ sufficiently small so that $C \delta \leq 1/4$. Then
\begin{align*}
(C \delta)^2 \geq |\nu_{x_1,t_1}- e_n| ^2 
&\geq (1-C\delta)^2 - 2 \nu_{x_1,t_1} \cdot e_n +1 
\geq (C \delta)^2 + \tfrac{1}{2},
\end{align*}
which is a contradiction. Hence, we conclude that
$$F_{1/2}(u) = \big\{ (x', x_n, t) \in Q_{1/2} : x_n = g(x',t) \big \}$$
for some continuous function $g: B_{1/2}'\times (-1/4, 0] \to (-\delta, \delta)$.
We need to show that $$g\in C^{1,\alpha}(B_{1/2}' \times (-1/4,0])$$ with appropriate estimates. It is enough to see that for any $(x_0',t_0) \in B_{1/2}' \times (-1/4,0]$, we have
$g\in C^{1,\alpha}(x_0',t_0)$ with $\|g\|_{C^{1,\alpha}(x_0',t_0)}\leq C_0$, where $C_0>0$ depends on $n$, $a_+$, and $a_-$. 
Indeed, let $(x_0',t_0) \in B_{1/2}' \times (-1/4,0]$ and take $(x_0,t_0)=(x_0',g(x_0'),t_0)\in F_{1/2}(u)$. By  Theorem~\ref{thm:improvflat2}, we have that for any  $(x,t) \in F_{1/2}(u)\cap Q_r(x_0,t_0)$,
 \begin{equation} \label{eq:estg1}
 |\nu_{x_0,t_0}\cdot (x-x_0)| \leq C \delta r^{1+\alpha} \quad \hbox{for all}~r\leq \bar r. 
 \end{equation}
 For convenience, we write $(\nu', \nu_n)=\nu_{x_0,t_0}$. Note that $|\nu'|\leq 1+ C\delta$ and $\nu_n\geq 1-C\delta\geq 1/2 > 0$. 
The left-hand side of \eqref{eq:estg1} can be written as
\begin{equation} \label{eq:estg2}
 |\nu_{x_0,t_0}\cdot (x-x_0)| = \nu_n \Big|g(x',t)-g(x_0',t_0) + \frac{\nu'}{\nu_n} \cdot (x'-x_0')\Big|.
\end{equation}
Take $(x,t)=(x',g(x'),t) \in F_{1/2}(u)$ with $t\leq t_0$, and set $r= 2 d_p((x,t),(x_0,t_0))$. Then  $(x,t)\in Q_{r}(x_0,t_0)$ and combining \eqref{eq:estg1} and \eqref{eq:estg2}, we get 
\begin{align*}
\Big |g(x',t)-g(x_0',t_0) + \frac{\nu'}{\nu_n} \cdot (x'-x_0')\Big| & \leq 2^{2+\alpha} C \delta d_p((x,t),(x_0,t_0))^{1+\alpha}.
\end{align*}
Moreover, we can estimate the parabolic distance by
$$
\big (|x-x_0|^2+|t-t_0| \big)^{\tfrac{1+\alpha}{2}} \leq 2^\alpha \big( |x'-x_0'|^2+ |t-t_0| \big)^\frac{1+\alpha}{2} + 2^\alpha |g(x',t)-g(x_0',t_0)|^{1+\alpha},
$$
where we used that $(a^2+b^2)^{p/2} \leq (a+b)^p \leq 2^{p-1}(a^p+b^p)$ for any $a,b\geq 0$ and $p\geq 1$.
Hence, it remains to control the term $|g(x',t)-g(x_0',t_0)|$. From the above estimates, we have
\begin{align*}
    |g(x',t)-g(x_0',t_0)|
     &\leq |\nu_{x_0,t_0}-e_n||x-x_0| + |\nu_{x_0,t_0}\cdot (x-x_0)|\\
     & \leq C\delta r +  C\delta r^{1+\alpha}
      \leq 4 C \delta d_p((x,t),(x_0,t_0))\\
      & \leq  4C\delta \big(|x'-x_0'|^2 + |t-t_0| \big)^{1/2}+  4C\delta  |g(x',t)-g(x_0',t_0)|.
\end{align*}
Taking $\delta$ sufficiently small so that $4 C \delta \leq 1/2$, we conclude that
\begin{align*}
\Big|g(x',t)-g(x_0',t_0) + \frac{\nu'}{\nu_n} \cdot (x'-x_0')\Big| \leq C_0 \delta \big(|x'-x_0'|^2 + |t-t_0|\big)^{\tfrac{1+\alpha}{2}}
\end{align*}
for some $C_0>0$ depending only on $n$, $a_+$, and $a_-$. This implies that $$\nabla' g(x_0',t_0)= - \frac{\nu'}{\nu_n},$$ 
and $g\in C^{1,\alpha}(x_0',t_0)$ with $\|g\|_{C^{1,\alpha}(x_0' ,t_0)} \leq C_0$. 
\end{proof}

\section{Higher regularity of the free boundary} \label{sec:hodograph}

To improve the regularity of the free boundary from $C^{1,\alpha}$ to smooth (Theorem~\ref{thm1}), we will use the Hodograph transform; see \cite{KN, KNS}. We start with some preliminary results.

\begin{lem} \label{globalreg}
 There exists $\bar \delta>0$ such that if $u$ is a viscosity solution to \eqref{FBP} in $Q_1$, satisfying \eqref{flatness}, then there is some $\alpha \in(0,1)$ and $C>0$, depending on $n$, $a_+$, and $a_-$, such that
$$
 \|\nabla u\|_{C^{\alpha}(\overline{Q_{1/2}})} \leq C.
$$
\end{lem}

\begin{proof}
By Remark~\ref{rem:normal} and Corollary~\ref{cor:normal}, we know that 
$\sup_{F_{1/2}(u)}|\nabla u|\leq C$ and
\begin{equation}\label{gradest}
|\nabla u(x_1,t_1) - \nabla u(x_2,t_2)| \leq C d_p((x_1,t_1),(x_2,t_2))^\alpha
\end{equation}
for any $(x_1,t_1),(x_2,t_2)\in F_{1/2}(u)$, and some $\alpha \in (0,1)$ and $C>0$ depending on $n$, $a_+$, and~$a_-$. We need to prove this estimate for all points in $Q_{1/2}$. 

Note that if $v$ is a solution to the heat equation 
\begin{equation} \label{eq:heat}
a w_t-\Delta w=0 \quad \hbox{in}~Q_{1/2} \quad \hbox{with}~a>0,
\end{equation}
then by classical interior estimates, we have that for all cylinders $Q_{2r}(x_0,t_0)\subset Q_{1/2}$, 
\begin{equation} \label{intest1}
    \sup_{Q_r(x_0,t_0)}|\nabla v|\leq \frac{C_0}{r} \sup_{Q_{2r}(x_0,t_0)} |v|.
\end{equation}
Moreover, there is some constant $\gamma\in (0,1)$, depending on $n$ and $a$, such that
\begin{equation} \label{intest2}
    \underset{Q_{r/2}(x_0,t_0)}{\osc} D_i v \leq \gamma \underset{Q_{r}(x_0,t_0)}{\osc} D_i v \leq \frac{2\gamma C_0}{r}  \sup_{Q_{2r}(x_0,t_0)} |v|
\end{equation}
for all $i=1,\hdots, n$, where we used \eqref{intest1} and the fact that $D_i v = \frac{\partial v}{\partial x_i}$ is also a solution of \eqref{eq:heat}. 
Therefore, by a simple triangle inequality argument, it is enough to prove \eqref{gradest} for any $(x_0,t_0)\in F_{1/2}(u)$ and $(x_1,t_1)\in Q_{1/2}^+(u)$ with $t_0=t_1$.

Let $r=2|x_1-x_0|$ so that $(x_1,t_1) \in Q_r(x_0,t_0)$.
Without loss of generality, we may assume that $r<\bar{r}/4$, where $\bar{r}$ is given in  Theorem~\ref{thm:improvflat2}.
Let $r_j=2^{-j} r$ and take a sequence of points $\{y_j\}_{j\geq 1} \subset  Q_{1/2}^+(u)$ with $y_1=x_1$, such that
\begin{equation} \label{points}
    |y_j-y_{j+1}|\leq r_j/2 \quad \hbox{and} \quad y_j \to x_0 \quad \hbox{as} \ j\to \infty.
\end{equation}
Consider the function $v=u-l_{x_0,t_0}$, where $l_{x_0,t_0}(x)=\nabla u(x_0,t_0)\cdot(x-x_0)$. Then by \eqref{intest2} and \eqref{points}, for any $i=1,\hdots,n$, we get
\begin{align*}
D_i v(x_1,t_1)- D_i v(x_0,t_0) & = \sum_{j=1}^\infty \big( D_i v(y_j,t_0)- D_i v(y_{j+1},t_0)\big)\\
&\leq \sum_{j=1}^\infty \underset{Q_{r_j/2}(y_j,t_0)}{\osc} D_i v \\
& \leq 2\gamma C_0 \sum_{j=1}^\infty \frac{1}{r_j} \sup_{Q_{2r_j}(y_j,t_0)}|v|\\
&\leq 2\gamma C_0  \sum_{j=1}^\infty \frac{1}{r_j} \sup_{Q_{4 r_j}(x_0,t_0)}|u-l_{x_0,t_0}|\\
& \leq  2\gamma C C_0 \delta \sum_{j=1}^\infty\frac{1}{r_j} (4 r_j)^{1+\alpha}=C_1 |x_1-x_0|^\alpha,
\end{align*}
where in the last inequality we used Theorem~\ref{thm:improvflat2}, since $4 r_j < \bar r$. Finally, we observe that $\nabla v(x_0,t_0)=0$ and $\nabla v(x_1,t_1)=\nabla u(x_1,t_1)-\nabla u(x_0,t_0)$. Therefore,
$$
|\nabla u(x_1,t_1)-\nabla u(x_0,t_0)| 
\leq C d_p((x_1,t_1),(x_0,t_0))^\alpha,
$$
where $C>0$ depends only on $n$, $a_+$, and $a_-$.
\end{proof}

\begin{lem} \label{monotonicity}
Under the previous assumptions, for any $\lambda \in (0,1/2)$, there is some $\rho \in (0,1/4)$ such that 
$
u_{n}(x,t) \geq \lambda >0 \, \hbox{for any}~(x,t)\in Q_\rho.
$
\end{lem}

\begin{proof}
For any $(x_0,t_0)\in F_{1/2}(u)$, taking $\delta>0$ sufficiently small, we know that 
$$
u_{n}(x_0,t_0) =  \nu_{x_0,t_0}\cdot e_n \geq 1- C\delta \geq 1/2>0.
$$
Without loss of generality, assume that $(0,0)\in F_{1/2}(u)$. 
Fix $\rho\in (0,1/4)$ to be chosen small. 
By Lemma~\ref{globalreg}, for any $(x,t)\in Q_\rho$, we have
$$
|u_{n}(x,t) - u_{n}(0,0)|\leq [\nabla u]_{C^{\alpha}(\overline{Q_{1/2}})} (|x|^2+|t|)^\alpha \leq C \rho^\alpha.
$$
Then for any $\lambda\in (0,1/2)$, we can choose $\rho$ sufficiently small such that
$$
u_{n}(x,t) \geq u_{n}(0,0) - C\delta \rho^\alpha \geq 1/2- C \rho^\alpha\geq \lambda >0.
$$
\end{proof}

By the previous lemma, $u$ is monotone in the $e_n$-direction in a parabolic cylinder $Q_{\rho}$, for some small $\rho\in(0,1/4)$. Consider the change of variables 
$$
 \psi(x,t)=(x',u(x,t),t) =(y,t), 
$$
which is well defined in $Q_\rho$ (i.e., $\psi$ is one-to-one). 
We define the Hodograph transform as the function $h : \psi(Q_\rho)\to (-\rho,\rho)$ given implicitly by
\begin{equation} \label{hodograph}
h(y,t)=h(\psi(x,t)))=x_n.
 \end{equation}
By construction, the free boundary of $u$ is a graph in the $e_n$-direction, parametrized by $h$ restricted to the set $\{y_n=0\}$. Moreover, their derivatives are related as follows:
\begin{align*}
u_t & = - \frac{h_t}{h_n}, \\
\nabla u & = -\frac{1}{h_n} (\nabla' h, - 1),\\
D^2 u & = - \frac{1}{h_n} A(\nabla h)^T D^2 h A (\nabla h),
\end{align*}
where $A(p) = (a_{ij}(p))_{1\leq i,j\leq n}$, with $a_{ij}=\delta_{ij}$ if $1\leq i,j<n$, $a_{in}(p)=0$, if $1\leq i<n$, $a_{nj}(p)=-p_j/p_n$ if $1\leq j<n$ and $a_{nn}(p)=1/p_n$, for $p\in \Rn$ with $p_n \neq 0$.

The following proposition shows that the free boundary problem \eqref{FBP} for $u$ becomes a nonlinear transmission problem for $h$. Moreover, $h$ inherits the regularity properties of $u$.

\begin{prop}[Properties of $h$] \label{proph}
Fix $\lambda \in (0,1/2)$. 
Let $h$ be the function given in \eqref{hodograph}. 
\begin{enumerate}[$(i)$]
\item (Regularity).~ There are $\sigma>0$ small and $C_0>0$ such that $h\in C^\infty(Q_\sigma \setminus \Gamma)\cap C^{1,\alpha}(\overline{Q_\sigma})$, where $\Gamma = Q_\sigma \cap \{y_n=0\}$, and 
$$
\|\nabla h\|_{C^{\alpha}(\overline{Q_\sigma})} \leq C_0 \quad \hbox{and} \quad h_{n}(y,t)\geq \lambda^{-1} >0.
$$
\item Let $B(p)=A(p)^TA(p)=(b_{ij}(p))_{1\leq i,j \leq n}$. Then $B(\nabla h)$ is uniformly elliptic in $Q_\sigma$, i.e.,
$$
\Lambda^{-1} |\xi|^2 \leq b_{ij}(\nabla h)\xi_i \xi_j  \leq  \Lambda |\xi|^2
$$
for all $\xi\in \R^n$ and $(y,t)\in Q_\sigma$, where $\Lambda >1$.
\item The function  $h$ satisfies the nonlinear transmission problem
\begin{equation} \label{TP}
\begin{cases}
a_+h_t - b_{ij}(\nabla h) D_{ij} h =0 & \hbox{in}~Q_{\sigma}^+,\\
 a_- h_t - b_{ij}(\nabla h) D_{ij} h =0 & \hbox{in}~Q_{\sigma^-} ,\\
 h_n^+ = h_n^- & \hbox{on}~\Gamma,
\end{cases}
\end{equation}
in the classical sense.
\end{enumerate}
\end{prop}

\begin{proof}
$(i)$ Since $\psi$ is continuous and $\psi(0,0)=(0,0)$, there exists some $\sigma>0$ small such that $Q_\sigma \subset \psi(Q_\rho)$.
Clearly, $u$ is smooth in $\{\pm u >0\}$. Hence, $h$ is smooth in $Q_\sigma^\pm$. By Lemma~\ref{monotonicity}, we know that 
$h_{n} = u_{n}^{-1} \geq \lambda^{-1}$ in $Q_\sigma$. 
 Moreover, by Lemma~\ref{globalreg}, 
 $$
 \| \nabla h\|_{C^{\alpha}(\overline{Q_\sigma})} \leq \lambda \|\nabla u\|_{C^{\alpha}(\overline{Q_\rho})} \leq C_0.
 $$
 
 $(ii)$ We have that 
 $$
 b_{ij}(p) =a_{ji}(p)a_{ij}(p) =
 \begin{cases}
 \delta_{ij} & \hbox{if} ~1\leq i,j<n,\\
 p_i^2/p_n^2 & \hbox{if}~1\leq i <n ~\hbox{and}~ j=n,\\
 p_j^2/p_n^2 & \hbox{if}~i=n ~\hbox{and}~ 1\leq j<n,\\
 (1+|p'|^2)/p_n^2 & \hbox{if}~i=j=n.
 \end{cases}
 $$
Hence, $b_{ij}(p)\xi_i\xi_j =|\xi'|^2 + 2\sum_{i=1}^{n-1} p_i^2/p_n^2 \xi_i \xi_n + (1+|p'|^2)/p_n^2 \xi_n^2$. 
By $(i)$, we see that
$$
\Lambda^{-1} |\xi|^2 \leq b_{ij}(\nabla h)\xi_i \xi_j  \leq  \Lambda |\xi|^2,
$$
 for some constant $\Lambda>1$ depending on $C_0$ and $\lambda$.

$(iii)$ From the above computations, it immediately follows that $h$ satisfies \eqref{TP} in the classical sense, since we already proved that $u$ is a classical solution of \eqref{FBP}.
\end{proof}

\begin{rem}
Observe that the local regularity of the free boundary $F(u)$ is equivalent to the local regularity of $h$ on the \textit{fixed} interface $\Gamma=Q_\sigma \cap  \{x_n=0\}$. Therefore, Theorem~\ref{thm1} is a straightforward consequence of the following result.
 \end{rem}

\begin{thm} 
The function $h$ given in \eqref{hodograph} is locally smooth in space and time on $\Gamma$.
\end{thm}

\begin{proof}
We may assume that $\sigma=1$. Otherwise, the function
$\sigma^{-1} h(\sigma y, \sigma^2 t)$ satisfies \eqref{TP} in $Q_1$. We need to show that $h|_\Gamma \in C^k_{loc}$ for any $k\geq 1$.
To this end, it is enough to see that for any $(y_0, t_0)\in Q_{1/2}\cap \{y_n=0\}$,
$$
D_{y'}^m D_t^l h|_\Gamma \in C^\alpha(y_0,t_0),
$$
for any multi-index $m\in \R^{n-1}$ and integer $l\geq 0$, where $\alpha\in (0,1)$ is given in Lemma~\ref{gradest}. For simplicity, we may assume that $(y_0,t_0)=(0,0)$. 
First, we establish the regularity in space. We claim that
\begin{equation} \label{improvregk}
D_{y'}^m h |_\Gamma \in C^{\alpha}(0,0) \quad\hbox{for all}~ |m|=k \ \hbox{and} \ k\geq 1.
\end{equation}
Indeed, the case $k=1$ follows from Proposition~\ref{proph}. In fact, we deduce the stronger regularity condition,  $\nabla h\in C^{\alpha}(\overline{Q_1})$ with $\|\nabla h\|_{C^{0,\alpha}(\overline{Q_1})}\leq C_0$ and $h_{n} \geq \lambda^{-1}>0$. 
Define 
\begin{equation} \label{aij}
a_{ij}(y,t) = b_{ij}(\nabla h(y,t)) \quad\hbox{for}~ (y,t)\in \overline{Q_1}.
\end{equation}
Then $a_{ij}\in C^{\alpha}(\overline{Q_1})$ with $\|a_{ij}\|_{C^\alpha(\overline{Q_1})}\leq  C$.
In particular, we are under the assumptions of Theorem~\ref{thm:bootstrap} with $f_+\equiv f_- \equiv 0$. Therefore, we have that
\begin{equation} \label{k=2}
h\in C^{2,\alpha}(\overline{Q_{1/2}^+}) \cap C^{2,\alpha}(\overline{Q_{1/2}^-}) 
\quad \hbox{with} \quad
\|h\|_{C^{2,\alpha}(\overline{Q_{1/2}^+})} + \|h\|_{C^{2,\alpha}(\overline{Q_{1/2}^-})}\leq C, 
\end{equation}
and thus, \eqref{improvregk} holds for $k=2$. It follows that
\begin{equation} \label{coef2}
a_{ij}\in C^{1,\alpha}(\overline{Q_{1/2}^+}) \cap C^{1,\alpha}(\overline{Q_{1/2}^-}) 
\quad \hbox{with} \quad
\|a_{ij}\|_{C^{1,\alpha}(\overline{Q_{1/2}^+})} + \|a_{ij}\|_{C^{1,\alpha}(\overline{Q_{1/2}^-})}\leq C.
\end{equation}
Let $w= D_{y'}^\mu h$ with $|\mu|=1$. Differentiating the equations in \eqref{TP}, we see that $w$ satisfies \eqref{limitpb} in $Q_{1/2}$, with $a_{ij}(y,t)$ as in \eqref{aij}, and right-hand sides given by
\begin{align*}
f_+ (y,t) & = D_{y'}^\mu a_{ij}(y,t) D_{ij} h^+(y,t) \quad\hbox{for}~(y,t)\in \overline{Q_{1/2}^+},\\
f_- (y,t) & = D_{y'}^\mu a_{ij}(y,t) D_{ij} h^-(y,t) \quad\hbox{for}~(y,t)\in \overline{Q_{1/2}^-},
\end{align*}
where $h^\pm = h|_{\overline{Q_{1/2}^\pm}}$.
By \eqref{k=2} and \eqref{coef2},  we have that $f_+ \in C^\alpha(\overline{Q_{1/2}^+})$ and $f_- \in C^\alpha(\overline{Q_{1/2}^-})$. Moreover, by taking the even reflection across $y_n=0$, we can extend the functions $f_+$ and $f_-$ in a $C^{0,\alpha}$-fashion to all of $Q_{1/2}$.
Applying Theorem~\ref{thm:bootstrap} to $w(x/2, t/4)$, we see that 
$$
w=D_{y'}^\mu h\in C^{2,\alpha}(\overline{Q_{1/4}^+}) \cap C^{2,\alpha}(\overline{Q_{1/4}^-}).
$$
Hence, \eqref{improvregk} holds for $k=3$. We observe that, at each step, we gain one more tangential derivative in space, which allows us to differentiate the equations in \eqref{TP} again. Therefore, iterating this argument, we conclude \eqref{improvregk} for any $k\geq 1$.

Next, we deal with the time regularity. In this case, the previous argument fails since we cannot differentiate \eqref{aij} with respect to $t$. Indeed, by \eqref{k=2}, we only know that
\begin{equation} \label{timereg}
    h\in C^{2,\alpha} \implies \nabla h \in C^{1,\alpha} \implies \nabla h \in C_t^{0,(1+\alpha)/2},
\end{equation}
which is not enough regularity to compute $D_t \nabla h$.
For $\tau \in (0,1/16)$ fixed, we consider the backward-in-time incremental quotients:
$$
\partial_\tau^{1/2} h (y,t) = \frac{h(y, t)- h(y, t-\tau)}{\tau^{1/2}} \quad \hbox{for}~(y,t)\in Q_{1/2}.
$$
By \eqref{k=2}, it is clear that  $\|\partial_\tau^{1/2} h\|_{L^\infty(Q_{1/2})} \leq C$.
Moreover, $\partial_\tau^{1/2} h $ satisfies \eqref{limitpb} in $Q_{1/4}$  with $a_{ij}(y,t)$ as in \eqref{aij}, and right-hand sides given by
\begin{align*}
f_+(y,t) & = \partial_\tau^{1/2} a_{ij} (y,t) D_{ij}h^+(y,t-\tau) \quad\hbox{for}~(y,t)\in \overline{Q_{1/4}^+},\\
f_-(y,t) & = \partial_\tau^{1/2} a_{ij} (y,t) D_{ij}h^-(y,t-\tau) \quad\hbox{for}~(y,t)\in \overline{Q_{1/4}^-}.
\end{align*}
By \eqref{coef2}, we know that $a_{ij}$ is piecewise $C^{1,\alpha}$ in $Q_{1/2}$. In particular, for any $(y_0,t_0)\in {Q_{1/2}^\pm}$, there is a linear polynomial $l_{y_0,t_0}$ such that
\begin{equation} \label{aijest}
    \sup_{(y,t)\in Q_{r}(y_0,t_0)\cap  Q_{1/2}^\pm} |a_{ij}(y,t)-l_{y_0,t_0}(y)| \leq C r^{1+\alpha}.
\end{equation}
We claim that $\partial_\tau^{1/2} a_{ij}$
is piecewise $C^{0,\alpha}$ in $Q_{1/4}$ with uniform estimates in $\tau$, i.e.,
\begin{equation} \label{aijreg}
  \|\partial_\tau^{1/2}   a_{ij}\|_{C^\alpha(\overline{Q_{1/4}^+})}+ \|\partial_\tau^{1/2}   a_{ij}\|_{C^\alpha(\overline{Q_{1/4}^-})}\leq C
\end{equation}
for some $C>0$ independent of $\tau$. 
It is enough to prove the estimate on $\overline{Q_{1/4}^+}$.
By \eqref{timereg},
$$
\sup_{(y,t)\in Q_{1/4}^+}| \partial_\tau^{1/2} a_{ij}(y,t)| 
\leq \|\nabla b_{ij}\|_\infty \sup_{(y,t)\in Q_{1/4}^+} \frac{|\nabla h(y,t)- \nabla h(y,t-\tau)|}{\tau^{1/2}} \leq C\tau^{\alpha/2}\leq C.
$$
Next, we control the $\alpha$-H\"{o}lder seminorm. 
Let $(y_1,t_1),(y_2,t_2)\in \overline{Q_{1/4}^+}$ and set $$r= 2 d_p((y_1,t_1),(y_2,t_2)).$$ 
Without loss of generality, suppose that $t_1 > t_2$ and $r\geq 4 \tau^{1/2}$ (taking $\tau>0$ sufficiently small). By a standard covering argument, we may also assume that $r \leq 8 \tau^{1/2}$.
Then testing \eqref{aijest} at the points $(y_1, t_1),(y_1,t_1-\tau), (y_2,t_2),(y_2,t_2-\tau) \in Q_r(y_1,t_1)\cap Q_{1/2}^+$, we get
\begin{align*}
\big |\partial_\tau^{1/2} a_{ij} (y_1,t_1) - \partial_\tau^{1/2} a_{ij} (y_2,t_2) \big| 
& = 
 \frac{|a_{ij}(y_1,t_1) - a_{ij}(y_1,t_1-\tau) - a_{ij}(y_2,t_2) + a_{ij}(y_2,t_2-\tau)|}{\tau^{1/2}} \\
& \leq  \sum_{k=1}^2  \frac{|a_{ij}(y_k,t_k) - l_{y_1,t_1}(y_k)|}{\tau^{1/2}}
+   \frac{| a_{ij}(y_k,t_k-\tau) - l_{y_1,t_1}(y_k) |}{\tau^{1/2}} \\
&\leq 4C \frac{r^{1+\alpha}}{\tau^{1/2}} \leq  C_1 d_p((y_1,t_1), (y_2,t_2))^\alpha,
\end{align*}
where $C_1>0$ does not depend on $\tau$. Hence, we proved \eqref{aijreg}. It follows that
$$
 \|f_+\|_{C^{0,\alpha}(\overline{Q_{1/4}^+})} \leq 4 \|\partial_\tau^{1/2} a_{ij} \|_{C^{0,\alpha}(\overline{Q_{1/4}^+})} \| D_{ij}h\|_{C^{0,\alpha}(\overline{Q_{1/4}^+})}\leq C.
$$
Analogously, we obtain a uniform $C^{0,\alpha}$ estimate for $f_-$. By
Theorem~\ref{thm:bootstrap}, we see that
$$
\partial_\tau^{1/2} h  \in C^{2,\alpha}(\overline{Q_{1/8}^+}) \cap C^{2,\alpha}(\overline{Q_{1/8}^-})
$$ 
with uniform $C^{2,\alpha}$ estimates in $\tau$. In particular, $\partial_\tau^{1/2} \nabla h$ is piecewise $C^{1,\alpha}$ and, arguing as before, we have $\partial_\tau^{1/2} \partial_\tau^{1/2} \nabla h$ is piecewise $C^{0,\alpha}$ uniformly in $\tau$. By compactness, we get that $D_t \nabla h$ is piecewise $C^{0,\alpha}$. We are now able to differentiate the equations in \eqref{TP} with respect to $t$, and apply once again Theorem~\ref{thm:bootstrap} to get 
$$
D_t h\in C^{2,\alpha}(\overline{Q_{1/16}^+}) \cap C^{2,\alpha}(\overline{Q_{1/16}^-}).
$$
Iterating this procedure, we get that $h|_\Gamma$ is locally smooth in time. 

Finally,  combining both bootstrap arguments (in space and in time), we conclude that the mixed derivatives of $h|_\Gamma$ are locally smooth. 
 \end{proof}

\section{{A counterexample to Lipschitz regularity}} \label{sec:counterexample}

A question raised in the introduction is: are solutions to our parabolic free transmission problem Lipschitz continuous in the parabolic sense? 
For convenience, we will now look at the problem in the form:
\begin{equation} \label{eq:main}
\partial_t c(u) - \Delta u=0, \quad c(s)=a_+ s^+-a_-s^-, \quad a_\pm>0, \quad a_+\neq a_-.
\end{equation}
By classical parabolic theory \cite{L}, thanks to its divergence form structure, we know that $u$ is locally $C^{0,\beta}$ for some $\beta \in (0,1).$ In this section, we will see that there exist special functions satisfying \eqref{eq:main} whose spacial gradient is not locally bounded, and thus, giving a negative answer to our initial question. 
More precisely, we consider homogeneous radial functions 
\begin{equation} \label{eq:selfsim}
    u(x,t) = (-t)^{\alpha/2} f\left(\frac{|x|}{\sqrt{-t}}\right) \qquad \hbox{for} \ (x,t) \in \R^n \times (-\infty,0),
\end{equation}
where $n\geq 3$, for some parameter $\alpha>0$ and self-similar profile $f: [0, \infty) \to \R$ such that  
\begin{equation}\label{eq:bounded}
s\mapsto s^{-\alpha} f(s)\quad \hbox{is bounded as}~s\to\infty. 
\end{equation}
Caffarelli and Stefanelli \cite[Theorem~1.1]{CS} established that locally bounded $\alpha$-homogeneous solutions to \eqref{eq:main}, with $a_+=1$ and $a_-=2$, exist only for some specific values of $\alpha$. Moreover, for the smallest possible value, $\alpha:=\alpha_1^-\in (0,2)$, we have that $f$ changes sign exactly one time, and
$$
u(0,t)= -(-t)^{\alpha_1^-/2} \quad \hbox{for}~t\in (-1,0).
$$
This implies that $u$ is not Lipschitz in time, but it could be Lipschitz in the parabolic sense, if $\alpha_1^-\geq 1$.
Note that, if we call $s=|x|/\sqrt{-t}$ the self-similar variable, then
\begin{equation*}
\nabla u(x,t) = s^{1-\alpha} f'(s) |x|^{2-\alpha} x
\end{equation*}
is locally bounded if $\alpha\geq 1.$ We will see that this bound for $\alpha$ is not true, in general, for arbitrary coefficients $a_\pm>0$.
By a scaling argument, we may assume that $a_+=1$ and $a_-:= \vep \in (0,1)$. 
Plugging in the expression \eqref{eq:selfsim} into the equation \eqref{eq:main}, we see that the profile $f$ satisfies the differential equations:
\begin{align} \label{eq:ODEs}
\begin{cases}
\displaystyle f''(s) + \left( \frac{n-1}{s} - \frac{s}{2} \right) f'(s) + \frac{\alpha}{2} f(s) =0, & \hbox{where}~f>0, \ s>0,\\[0.5cm]
\displaystyle f''(s) + \left( \frac{n-1}{s} - \frac{\vep s}{2} \right) f'(s) + \frac{\vep \alpha}{2} f(s) =0, & \hbox{where}~f<0, \ s>0.
\end{cases}
\end{align}
Moreover, we prescribe the boundary conditions:
\begin{equation}\label{eq:bc}
 f(0)=1 \quad  \hbox{and} \quad f'(0)=0.
\end{equation} 
To solve the ODEs, we consider the \textit{confluent hypergeometric functions} $M(a,b,z)$ (Kummer function) and $U(a,b,z)$ (Tricomi function) satisfying the differential equation \cite{T}:
$$
z g''(z) + (b-z)g'(z) - a g(z) =0, \quad z\in \R.
$$
The function $M(a,b,z)$ is well-defined for any $a\in \R$ and $b\in \R\setminus \Z_-$, and it is given by
$$
M(a,b,z)= \sum_{k=0}^\infty \frac{(a)_k }{(b)_k k!} z^k = 1+ \frac{a}{b} z + \frac{a(a+1)}{b(b+1)2!} z^2 + \cdots
$$
Hence, by definition, $M(a,b,0)=1$ for any $a\in \R$ and $b\in \R\setminus \Z_-$. Moreover, for any $a, b\in \R$, the function $U(a,b,z)$ is uniquely determined by the property: 
$$
z \mapsto z^a U(a,b,z)  \quad \hbox{is bounded as}~ z\to\infty.
$$ 
Notice that, by a simple change of variables, the functions
\begin{equation}\label{eq:kummer}
M\Big(-\frac{\alpha}{2}, \frac{n}{2}, \frac{s^2}{4}\Big) \quad \hbox{and} \quad
 U\Big(-\frac{\alpha}{2}, \frac{n}{2}, \frac{\vep s^2}{4}\Big)
 \end{equation}
 satisfy the first and second equations of \eqref{eq:ODEs}, respectively.

We need the following lemma regarding the zeros of $M$ and $U$ \cite{G,T}.

\begin{lem}[Properties of zeros]  \label{lem:zeros}
Fix $a \in [-1,0)$ and $b\geq1$.
\begin{enumerate}[$(i)$]
\item (Uniqueness). The functions $M(a,b,z)$ and $U(a,b,z)$ have a unique positive zero. 
\item (Monotonicity). Let $z_M(a,b)$ and $z_U(a,b)$ be the unique zeros from part $(i)$. Then $z_M(\cdot,b)$ is strictly increasing and $z_U(\cdot,b)$ is strictly decreasing with respect to $a$.
\end{enumerate}
\end{lem}

For $n\geq 3$, $\alpha\in (0,2]$, and  $\vep\in (0,1)$ fixed, let 
$$\bar z_M(\alpha, n) \quad \hbox{and}   \quad  \bar z_U^\vep (\alpha, n)$$ 
be the unique positive zeros of the functions $M$ and $U$ given in \eqref{eq:kummer}.
By Lemma~\ref{lem:zeros} $(ii)$, with $a=-\alpha/2$ and $b=n/2$, it follows that
\begin{equation} \label{eq:monotonicity}
\bar z_M(\alpha,n) \nearrow \infty\quad \hbox{and} \quad  \bar z_U^\vep (\alpha,n)\searrow 0 \qquad \hbox{as } \alpha \searrow 0.
\end{equation}
Indeed, by definition of $M$, we have that $M(0,n/2,s^2/4)=1$ for any $s\geq0$. Hence, $M$ does not have any zeros when $\alpha=0$. By continuity of $M$ with respect to $\alpha$ and the monotonicity of $\bar z_M(\cdot, n)$ we get that $\bar z_M(\alpha,n) \nearrow \infty$ as $\alpha \searrow 0.$ A similar reasoning shows that $\bar z_U^\vep (\alpha,n)\searrow 0$.

We observe that when $\alpha=2$, these functions are polynomials of degree 2:
$$
M\Big(-1, \frac{n}{2}, \frac{s^2}{4}\Big) = -\frac{s^2}{2n} +1 \quad \hbox{and} \quad  
U\Big(-1, \frac{n}{2}, \frac{\vep s^2}{4}\Big) = \frac{\vep s^2}{4} -\frac{n}{2}.
$$
In this case,  we have that
\begin{equation}\label{eq:zeros}
\bar z_M(2,n) = \sqrt{2n} < \sqrt{2n/\vep} = \bar z_U^\vep(2,n).
\end{equation}
 By \eqref{eq:monotonicity} and \eqref{eq:zeros}, we see that for $n\geq 3$ and $\vep\in (0,1)$ fixed, there is some $\alpha\in (0,2)$ such that $\bar z_M(\alpha, n)=\bar z_U^\vep(\alpha, n)=:s_{\vep} \in (\sqrt{2n}, \sqrt{2n/\vep})$, and the solution $f(s)$ to \eqref{eq:ODEs} satisfying \eqref{eq:bounded} and \eqref{eq:bc} is given explicitly by
 \begin{align}\label{eq:profilef}
f(s)= \left\{
\begin{array}{cl}
\displaystyle M\Big(-\frac{\alpha}{2}, \frac{n}{2}, \frac{s^2}{4}\Big) & \hbox{if}~s\in [0, s_\vep],\\[0.3cm]
\displaystyle - U\Big(-\frac{\alpha}{2}, \frac{n}{2}, \frac{\vep s^2}{4}\Big) & \hbox{if}~ s \in (s_\vep, \infty).\\
\end{array}\right.
\end{align}
\medskip

Furthermore, the following theorem shows that for $\vep$ sufficiently small, we have $\alpha<1$.

\begin{thm} \label{thm:counterxample}
   For all $n\geq 3$, there is some $\vep_0=\vep_0(n)<1$ such that if $\vep \in(0, \vep_0)$, then there exists $\alpha=\alpha(n,\vep)<1$ such that
    $
    \bar z_M(\alpha, n) = \bar z_U^\vep (\alpha, n).
    $
    Moreover, $\alpha\to 0$ as $\vep\to 0$.
\end{thm}

\begin{proof}
Let $n\geq 3$ and $\vep\in (0,1)$. Note that if $\bar z_U (\alpha, n)$ is the unique positive solution of
$$
U\Big(-\frac{\alpha}{2}, \frac{n}{2}, \frac{s^2}{4} \Big) =0,
$$
then $\bar z_U (\alpha, n) = \sqrt{\vep} \bar z_U^\vep (\alpha, n)$.
Moreover, by \eqref{eq:zeros} and the monotonicity in \eqref{eq:monotonicity} with $\vep=1$, we have 
$$
 \bar z_M (1, n)> \bar z_M(2,n)=\bar z_U(2,n)> \bar z_U (1, n).
$$
Define $\vep_0=\vep_0(n) <1$ as 
$${\vep_0} = \left(\frac{\bar z_U(1,n)}{\bar z_M(1,n)}\right)^2.$$ 
Then for any $\vep \in (0,\vep_0)$, we have 
$$
\bar z_M (1, n) = \frac{\bar z_U(1,n)}{\sqrt{\vep_0}} < \frac{\bar z_U(1,n)}{\sqrt{\vep}}=\bar z_U^\vep(1,n).
$$
Using again the monotonicity, we conclude that there exists some $\alpha=\alpha(n,\vep)<1$ such that
$$
    \bar z_M(\alpha, n) = \bar z_U^\vep (\alpha, n).
$$
Moreover, it follows that
$$
\frac{\bar z_U (\alpha, n)}{\bar z_M(\alpha, n)} = \sqrt{\vep} \longrightarrow 0 \qquad \hbox{as}~\vep \to 0.
$$
Therefore, for any $n\geq3$ fixed, 
$$
\bar z_U (\alpha, n)\longrightarrow 0 \quad \hbox{or} \quad \bar z_M(\alpha, n)\longrightarrow \infty\qquad \hbox{as}~ \vep\to 0. 
$$
By \eqref{eq:monotonicity}, we must have that $\alpha\to 0$ as $\vep\to0$.
\end{proof}

We illustrate the existence of such $\alpha<1$ for $n=3$ and $n=10$ in Figure~\ref{fig:plots}.
\begin{figure}[h]
\includegraphics[scale=0.6]{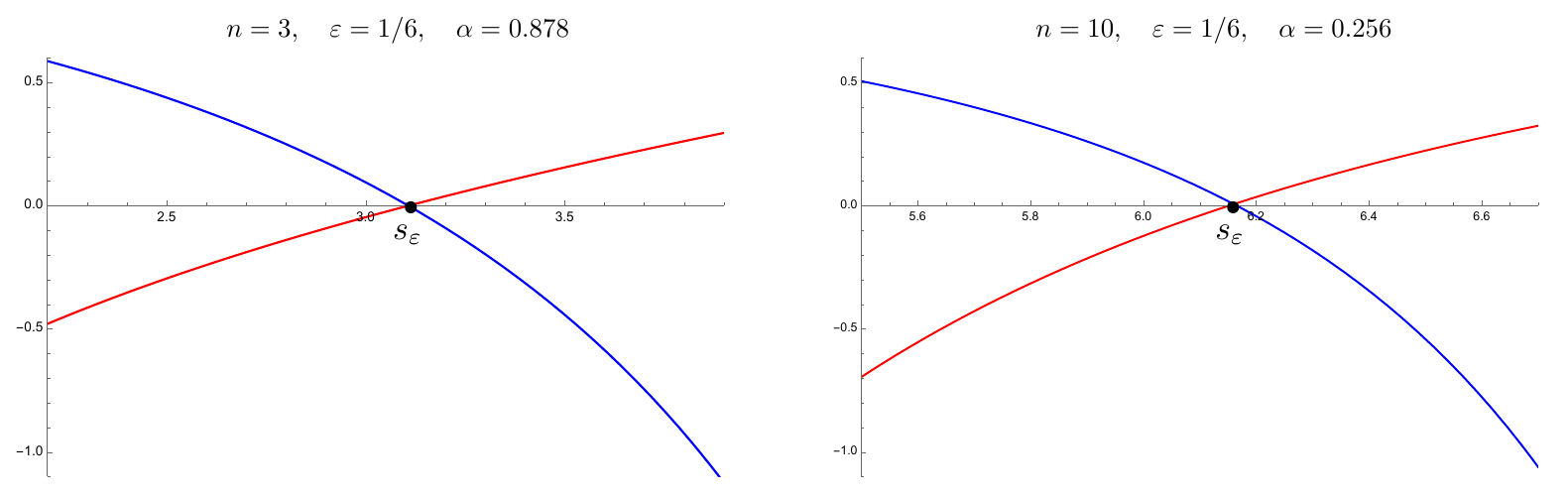}
\caption{Graphs of $M$ (blue) and $U$ (red) intersecting at the unique zero $s_\vep$.}
\label{fig:plots}
\end{figure}

We conclude from Theorem~\ref{thm:counterxample} that for $n\geq 3$, there is some coefficient $\vep\in (0,1)$ such that the solution $u$ to \eqref{eq:main} with $a_+=1$ and $a_-=\vep$, defined in \eqref{eq:selfsim} and \eqref{eq:profilef} for some $\alpha <1$, is not Lipschitz continuous, and thus, we obtain the desired counterexample.

\section{Appendix} \label{sec:appendix}

Given two nonempty sets $X, Y \subset \R^{m}$, the \textit{Hausdorff distance} is defined as
 $$
 d_{\h}(X,Y) = \inf \big \{\vep >0 : X \subset B_\vep(Y) \ \hbox{and}\ Y \subset B_\vep(X) \big\},
$$
where $B_\vep (Z) = \{ x \in \R^m : d(x,Z) \leq \vep\}$ and $d(x,Z)=\inf_{z\in Z} |x-z|$. 

Let $(M,d)$ be a metric space and $\mathcal{K}(M)$ be the set of all nonempty compact subsets of $M$. Then $(\mathcal{K}(M), d_{\h})$ is a metric space. Moreover,  if $M$ is compact, then $\mathcal{K}(M)$ is also compact.

\begin{lem} \label{lem:appendix}
Let $\{v_k\}_{k\geq 1}$ be a sequence of continuous real functions on $\overline{B_1}\subset \Rn$ satisfying:
\begin{enumerate}[$(i)$]
\item $|v_k|\leq 1$ on $\overline{B_1}$ for all $k\geq 1$;
\item There is a sequence $\delta_k\to 0$ and a modulus of continuity $\omega: [0,\infty)\to[0,\infty)$ such that 
$$
|v_k(x)-v_k(y)|\leq \omega(|x-y|)
$$
for all $x,y\in \overline{B_1}$ with $|x-y|\geq \delta_k$.
\end{enumerate} 
For each $k\geq 1$, let $G(v_k) =\{ (x,p) : x\in \overline{B_1}, \ p= v_k(x)\}$ be the graph of $v_k$ on $\overline{B_1}$. Then there exists a continuous function $v: \overline{B_1} \to [-1,1]$ such that
$$
d_{\h} (G(v_k), G(v)) \to 0 \quad \hbox{as}~ k \to \infty.
$$
Moreover, for all $x,y\in \overline{B_1}$, we have
$$
|v(x)-v(y)|\leq \omega(|x-y|).
$$

\end{lem}

\begin{proof}
By the continuity of $v_k$ and $(i)$, we have that $G_k\equiv G(v_k)$ is a compact subset of $\overline{B_2}$. Since $\mathcal{K}(\overline{B_2})$ is a compact metric space with the Hausdorff distance, there exists a compact set $K\subset \overline{B_2}$ such that
\begin{equation}\label{convhauss}
d_{\h}(G_k, K)\to 0 \quad\hbox{as}~ k \to \infty,
\end{equation}
up to a subsequence.
Next, we show that $K$ must be the graph of some function. If not, there exist two points $(x,p), (x,q) \in K$ with $p\neq q$. By \eqref{convhauss} for all $\vep>0$, there is $k_0\geq 1$ such that $(x,p) \in B_\vep(G_k)$, for all $k\geq k_0$, i.e.,
$d((x,p), G_k) \leq \vep$ for all  $k\geq k_0.$
In particular, since $G_k$ is compact, the distance is attained at some point $(x_k, p_k)\in G_k$. Hence,
\begin{equation*}
(x_k, p_k) \to (x,p) \quad\ \hbox{as}~ k \to \infty.
\end{equation*}
Similarly, there are points $(y_k, q_k)\in G_k$ such that
$$
(y_k, q_k) \to (x,q) \quad\ \hbox{as}~ k \to \infty.
$$
We distinguish two cases: If $|x_k-y_k|\geq \delta_k$ for infinitely many $k$'s, then by $(ii)$, we get
\begin{align*}
|p_k-q_k| &=  |v_k(x_k)-v_k(y_k)| \leq  \omega (|x_k-y_k|).
\end{align*}
Letting $k\to\infty$, we get that $p=q$, which is a contradiction.
Otherwise, there is some $z_k\in \overline{B_1}$ such that $\delta_k \leq |x_k-z_k|\leq 3\delta_k$ and $\delta_k \leq |y_k - z_k|\leq 3\delta_k$.
Then
\begin{align*}
|p_k-q_k| &\leq  |v_k(x_k)-v_k(z_k)| +  |v_k(z_k)-v_k(y_k)| \\
&\leq  \omega(|x_k-z_k|) + \omega(|z_k-y_k|)\\
&\leq 2 \omega (3\delta_k).
\end{align*}
Letting $k\to\infty$, we get the contradiction.

Observe that for any $x\in \overline{B_1}$, there is some $p\in [-1,1]$ such that $(x,p)\in K$. Indeed, let $x\in \overline{B_1}$ and consider $(x,v_k(x))\in G_k$. By \eqref{convhauss}, for all $\vep>0$, there is $k_0\geq 1$ such that $(x,v_k(x)) \in B_\vep(K)$ for all $k\geq k_0$. Taking $\vep= 1/j$, by a diagonalization argument, we obtain that there is a subsequence of $\{(x,v_k(x))\}_{k\geq 1}$ converging to $(x,p)\in K$, where $|p|\leq 1$ by $(i)$.

We define the function $v: \overline{B_1}\to K$ as $v(x)=p$.  Let $x, y\in \overline{B_1}$ with $x\neq y$. Then
$$
|v(x)-v(y)| = \lim_{k\to\infty} |v_k(x) - v_k(y)| \leq \omega(|x-y|),
$$
where we used that $|x-y|\geq \delta_k>0$ for $k$ large enough.
\end{proof}



\begin{thebibliography}{100}


\bibitem{AL} H. W. Alt and S. Luckhaus.
Quasilinear elliptic-parabolic differential equations.
\textit{Math. Z.} 183 (1983), no.3, 311--341.

\bibitem{AT}M.~D. Amaral and E.~V. Teixeira.
Free transmission problems.
\textit{Comm. Math. Phys.} 337 (2015), no.3, 1465--1489. 

\bibitem{AM} J. Andersson and H. Mikayelyan.
The zero level set for a certain weak solution, with applications to the Bellman equations.
\textit{Trans. Amer. Math. Soc.} 365 (2013), no.5, 2297--2316.

\bibitem{AW} J. Andersson and G. S. Weiss.
A parabolic free boundary problem with Bernoulli type condition on the free boundary.
\textit{J. Reine Angew. Math.} 627 (2009), 213--235.

\bibitem{ACS} I. Athanasopoulos, L. A. Caffarelli, and S. Salsa.
Regularity of the free boundary in parabolic phase-transition problems.
\textit{Acta Math.} 176 (1996), no.2, 245--282.

\bibitem{ACS1}  I. Athanasopoulos, L. A. Caffarelli, and S. Salsa.
Caloric functions in Lipschitz domains and the regularity of solutions to phase transition problems.
\textit{Ann. of Math. (2)} 143 (1996), no.3, 413--434.

\bibitem{ACS2} I. Athanasopoulos, L. A. Caffarelli, and S. Salsa.
Phase transition problems of parabolic type: flat free boundaries are smooth.
\textit{Comm. Pure Appl. Math.} 51(1998), no.1, 77--112.


\bibitem{BD} B. Bazaliy and S. Degtyarev.
Classical solutions of many-dimensional elliptic-parabolic free boundary problems.
\textit{Nonlinear Differential Equations Appl.} 16 (2009), no.4, 421--443.

\bibitem{BH} M. Bertsch and J. Hulshof.
Regularity results for an elliptic-parabolic free boundary problem.
\textit{Trans. Amer. Math. Soc.} 297 (1986), no.1, 337--350.

\bibitem{C1} L. A. Caffarelli.
A Harnack inequality approach to the regularity of free boundaries. I. Lipschitz free boundaries are $C^{1,\alpha}$.
\textit{Rev. Mat. Iberoamericana} 3 (1987), no.2, 139--162.

\bibitem{C2} L. A. Caffarelli.
A Harnack inequality approach to the regularity of free boundaries. II. Flat free boundaries are Lipschitz.
\textit{Comm. Pure Appl. Math.} 42 (1989), no.1, 55--78.


\bibitem{CC} L. A. Caffarelli and X. Cabr\'{e}.
\textit{Fully nonlinear elliptic equations.}
Amer. Math. Soc. Colloq. Publ., 43.
American Mathematical Society, Providence, RI, 1995.

\bibitem{CDSS} L. Caffarelli, D. De Silva, and O. Savin.
Two-phase anisotropic free boundary problems and applications to the Bellman equation in 2D.
\textit{Arch. Ration. Mech. Anal.} 228 (2018), no.2, 477--493.



\bibitem{CLW} L. A. Caffarelli, C. Lederman, and N. Wolanski.
Uniform estimates and limits for a two phase parabolic singular perturbation problem.
\textit{Indiana Univ. Math. J.} 46 (1997), no.2, 453--489.

\bibitem{CSa}  L. A. Caffarelli  and S. Salsa.
\textit{A geometric approach to free boundary problems.}
Grad. Stud. Math., 68.
American Mathematical Society, Providence, RI, 2005. 

\bibitem{CS} L. A. Caffarelli and U. Stefanelli.
A counterexample to $C^{2,1}$ regularity for parabolic fully nonlinear equations.
\textit{Comm. Partial Differential Equations} 33 (2008), no.7-9, 1216--1234.

\bibitem{CV} L. A. Caffarelli and J. L. V\'{a}zquez.
A free-boundary problem for the heat equation arising in flame propagation.\textit{Trans. Amer. Math. Soc.} 347 (1995), no.2, 411--441.

\bibitem{CFS} M. C. Cerutti, F. Ferrari, and S. Salsa.
Two-phase problems for linear elliptic operators with variable coefficients: Lipschitz free boundaries are $C^{1,\gamma}$.
\textit{Arch. Ration. Mech. Anal.} 171 (2004), no.3, 329--348.

\bibitem{D} D. Danielli.
A singular perturbation approach to a two-phase parabolic free boundary problem arising in flame propagation.
Thesis (Ph.D.)--Purdue University.
\textit{ProQuest LLC}, Ann Arbor, MI, 1999.

\bibitem{DS} D. De Silva. 
Free boundary regularity for a problem with right hand side. 
\textit{Interfaces Free Bound.} 13  (2011), no.2, 223--238.

\bibitem{DSFS}D.  De Silva, N. Forcillo and O. Savin.
Perturbative estimates for the one-phase Stefan problem.
\textit{Calc. Var. Partial Differential Equations} 60 (2021), no.6. 

\bibitem{DBG} E. DiBenedetto and R. Gariepy.
Local behavior of solutions of an elliptic-parabolic equation.
\textit{Arch. Rational Mech. Anal.} 97 (1987), no.1, 1--17.

\bibitem{EL} L. C. Evans and S. Lenhart.
The parabolic Bellman equation.
\textit{Nonlinear Anal.} 5 (1981), no.7, 765--773.

\bibitem{G}L. Gatteschi.
\textit{New inequalities for the zeros of confluent hypergeometric functions.}
Asymptotic and computational analysis (Winnipeg, MB, 1989), 175--192.
Lecture Notes in Pure and Appl. Math., 124
Marcel Dekker, Inc., New York, 1990.

\bibitem{HW} J. Hulshof and N. Wolanski.
Monotone flows in $N$-dimensional partially saturated porous media: Lipschitz-continuity of the interface.
\textit{Arch. Rational Mech. Anal.} 102 (1988), no.4, 287--305.

\bibitem{KLS} S. Kim, K-A. Lee, and H. Shahgholian.
An elliptic free boundary arising from the jump of conductivity.
\textit{Nonlinear Anal.} 161 (2017), 1--29.

\bibitem{KP} I. C. Kimand N. Po\v{z}\'{a}r.
Nonlinear elliptic-parabolic problems.
\textit{Arch. Ration. Mech. Anal.} 210 (2013), no.3, 975--1020.

\bibitem{KN} D. Kinderlehrer and L. Nirenberg.
Regularity in free boundary problems.
\textit{Ann. Scuola Norm. Sup. Pisa Cl. Sci. (4)} 4 (1977), no.2, 373--391.

\bibitem{KNS} D. Kinderlehrer, L. Nirenberg, and J. Spruck.
Regularity in elliptic free boundary problems.
\textit{J. Analyse Math.} 34 (1978), 86--119.


\bibitem{LN} Y. Y. Li and L. Nirenberg.
Estimates for elliptic systems from composite material.
Dedicated to the memory of J\"{u}rgen K. Moser.
\textit{Comm. Pure Appl. Math.} 56 (2003), no.7, 892--925.

\bibitem{LV} Y. Y. Li and and M. Vogelius.
Gradient estimates for solutions to divergence form elliptic equations with discontinuous coefficients.
\textit{Arch. Ration. Mech. Anal.} 153 (2000), no.2, 91--151.

\bibitem{L} G. M. Lieberman.
\textit{Second order parabolic differential equations.}
World Scientific Publishing Co., Inc., River Edge, NJ, 1996.

\bibitem{MS}
 E. Milakis and L. E. Silvestre.
Regularity for fully nonlinear elliptic equations with Neumann boundary data.
\textit{Comm. Partial Differential Equations} 31 (2006), no.~7-9, 1227--1252.


\bibitem{PS} E. A. Pimentel and M. S. Santos.
Fully nonlinear free transmission problems.
\textit{Interfaces Free Bound.} 25 (2023), no.~3, 325--342.

\bibitem{PSw} E. Pimentel and A. \'{S}wi\c{e}ch.
Existence of solutions to a fully nonlinear free transmission problem.
\textit{J. Differential Equations} 320 (2022), 49--63.


\bibitem{ST} N. Soave and S. Terracini.
An anisotropic monotonicity formula, with applications to some segregation problems.
\textit{J. Eur. Math. Soc.} 25 (2023), no.9, 3727--3765.

\bibitem{T} F. G. Tricomi.
\textit{Fonctions hyperg\'eom\'etriques confluentes.}
M\'emor. Sci. Math., Fasc. CXL.
Gauthier-Villars, Paris, 1960.

\bibitem{LW} L. Wang.
On the regularity theory of fully nonlinear parabolic equations. I.
\textit{Comm. Pure Appl. Math.} 45 (1992), no.1, 27--76.

\bibitem{W1} P-Y. Wang.
Regularity of free boundaries of two-phase problems for fully nonlinear elliptic equations of second order. I. Lipschitz free boundaries are $C^{1,\alpha}$.
\textit{Comm. Pure Appl. Math.} 53 (2000), no.7, 799--810.


\bibitem{W2} P-Y. Wang.
Regularity of free boundaries of two-phase problems for fully nonlinear elliptic equations of second order. II. Flat free boundaries are Lipschitz.
\textit{Comm. Partial Differential Equations} 27 (2002), no.7-8, 1497--1514.

\bibitem{WLSXMZ}C. Wang, X. Liu, J. Shao, W. Xiong, W. Ma, and Y. Zheng.
Structural transition and temperature-driven conductivity switching of single crystalline VO2(A) nanowires.
\textit{RSC Adv.} 4 (2014), 64021--64026.


\bibitem{W} G. S. Weiss.
A singular limit arising in combustion theory: fine properties of the free boundary.
\textit{Calc. Var. Partial Differential Equations} 17 (2003), no.3, 311--340.

\end{thebibliography}
\end{document}